\documentclass{amsart}

\usepackage{textcmds} 
\usepackage{psfrag}
\usepackage{amscd}
\usepackage{rotating}
\usepackage{lscape}
\usepackage{amsbsy}
\usepackage{verbatim}
\usepackage{moreverb}
\usepackage{url}
\usepackage{cancel}
\usepackage{amsfonts}
\usepackage{amsmath}
\usepackage{amssymb}
\usepackage{amsthm}
\usepackage{amsxtra}
\usepackage{amscd}
\usepackage{bbm}
\usepackage[utf8]{inputenc}

\usepackage[doi=false,url=false]{biblatex}

\usepackage{comment}
\usepackage{eucal}
\usepackage{extarrows}
\usepackage[centering,bottom=3cm,margin=2.7cm,top=3cm]{geometry}
\usepackage{hyperref}
\usepackage{mathrsfs}
\usepackage{pictexwd,dcpic}
\usepackage{setspace}
\usepackage{tikz-cd}
\usepackage{xcolor}
\usepackage{wasysym}
\usepackage[all]{xy}
\usepackage{graphicx}
\usepackage{epsfig}
\usepackage[foot]{amsaddr}
\usepackage{enumitem}

\usetikzlibrary{matrix,arrows,decorations.pathmorphing}

\newcommand\myshade{85}
\colorlet{mylinkcolor}{violet}
\colorlet{mycitecolor}{blue}
\colorlet{myurlcolor}{green}

\hypersetup{
  linkcolor  = mylinkcolor!\myshade!black,
  citecolor  = mycitecolor!\myshade!black,
  urlcolor   = myurlcolor!\myshade!black,
  colorlinks = true,
}

\newcommand{\ssec}[0]{\subsection}
\newcommand{\sssec}[0]{\subsubsection}

\newcommand{\on}{\operatorname}
\newcommand{\nc}{\newcommand}
\nc{\TV}{To\"en--Vezzosi}
\nc{\et}{{\on{\acute{e}t}}}
\nc{\virg}[1]{``#1"}
\nc{\bigt}[1]{\bigl( #1 \bigr) }
\nc{\Bigt}[1]{\Bigl( #1 \Bigr) }
\nc{\wt}{\tilde}
\nc{\ue}{\textup{e}}
\nc{\uh}{\textup{h}}
\renewcommand{\tilde}{\widetilde}
\renewcommand{\hat}{\widehat}
\theoremstyle{plain}
\newtheorem{thm}[subsubsection]{Theorem}
\newtheorem{mainthm}{Theorem}

\newtheorem{lem}[subsubsection]{Lemma}
\newtheorem{prop}[subsubsection]{Proposition}
\newtheorem{cor}[subsubsection]{Corollary}

\theoremstyle{definition}

\newtheorem{quest}[subsubsection]{Question}

\newtheorem{notat}[subsubsection]{Notation}

\theoremstyle{remark}
\newtheorem{rmk}[subsubsection]{Remark}

\newtheorem{rem}[subsubsection]{Remark}

\newcommand{\thmref}[1]{Theorem~\ref{#1}}

\newcommand{\lemref}[1]{Lemma~\ref{#1}}
\newcommand{\propref}[1]{Proposition~\ref{#1}}
\newcommand{\corref}[1]{Corollary~\ref{#1}}

\newcommand{\remref}[1]{Remark~\ref{#1}}
\newcommand{\questref}[1]{Question~\ref{#1}}

\numberwithin{equation}{section}
\theoremstyle{plain}
\newtheorem*{unthm}{Theorem}

\theoremstyle{definition}
\newcommand{\ccC}[0]{\mathcal{C}}

\newcommand{\ccF}[0]{\mathcal{F}}
\newcommand{\ccG}[0]{\mathcal{G}}
\newcommand{\ccH}[0]{\mathcal{H}}

\newcommand{\ccL}[0]{\mathcal{L}}
\newcommand{\ccM}[0]{\mathcal{M}}

\newcommand{\ccO}[0]{\mathcal{O}}

\newcommand{\ccR}[0]{\mathcal{R}}
\newcommand{\ccS}[0]{\mathcal{S}}

\newcommand{\bbA}[0]{\mathbb{A}}

\newcommand{\bbC}[0]{\mathbb{C}}
\newcommand{\bbD}[0]{\mathbb{D}}
\newcommand{\bbE}[0]{\mathbb{E}}
\newcommand{\bbF}[0]{\mathbb{F}}

\newcommand{\bbP}[0]{\mathbb{P}}
\newcommand{\bbQ}[0]{\mathbb{Q}}

\newcommand{\bbU}[0]{\mathbb{U}}

\newcommand{\bbZ}[0]{\mathbb{Z}}

\newcommand{\ffF}[0]{\mathfrak{F}}

\newcommand{\ffQ}[0]{\mathfrak{Q}}

\newcommand{\ffS}[0]{\mathfrak{S}}

\nc{\ffa}{{\mathfrak{a}}}
\nc{\ffb}{{\mathfrak{b}}}
\nc{\ffc}{{\mathfrak{c}}}
\nc{\ffd}{{\mathfrak{d}}}
\nc{\ffe}{{\mathfrak{e}}}
\nc{\fff}{{\mathfrak{f}}}
\nc{\ffg}{{\mathfrak{g}}}
\nc{\ffgl}{{\mathfrak{gl}}}
\nc{\ffh}{{\mathfrak{h}}}
\nc{\ffii}{{\mathfrak{i}}}
\nc{\ffj}{{\mathfrak{j}}}
\nc{\ffk}{{\mathfrak{k}}}
\nc{\ffl}{{\mathfrak{l}}}
\nc{\ffm}{{\mathfrak{m}}}
\nc{\ffn}{{\mathfrak{n}}}
\nc{\ffo}{{\mathfrak{o}}}
\nc{\ffp}{{\mathfrak{p}}}
\nc{\ffq}{{\mathfrak{q}}}
\nc{\ffr}{{\mathfrak{r}}}
\nc{\ffs}{{\mathfrak{s}}}
\nc{\fft}{{\mathfrak{t}}}
\nc{\ffu}{{\mathfrak{u}}}
\nc{\ffv}{{\mathfrak{v}}}
\nc{\ffx}{{\mathfrak{x}}}
\nc{\ffy}{{\mathfrak{y}}}
\nc{\ffw}{{\mathfrak{w}}}
\nc{\ffz}{{\mathfrak{z}}}
\newcommand{\rmB}[0]{\mathrm{B}}

\newcommand{\scrE}[0]{\mathscr{E}}

\newcommand{\scrG}[0]{\mathscr{G}}

\newcommand{\scrO}[0]{\mathscr{O}}

\nc{\sA}{{\mathsf{A}}}
\nc{\sB}{{\mathsf{B}}}
\nc{\sC}{{\mathsf{C}}}
\nc{\sD}{{\mathsf{D}}}
\nc{\sE}{{\mathsf{E}}}
\nc{\sF}{{\mathsf{F}}}
\nc{\sG}{{\mathsf{G}}}
\nc{\sK}{{\mathsf{K}}}
\nc{\sM}{{\mathsf{M}}}
\nc{\sN}{{\mathsf{N}}}
\nc{\sO}{{\mathsf{O}}}
\nc{\sW}{{\mathsf{W}}}
\nc{\sQ}{{\mathsf{Q}}}
\nc{\sP}{{\mathsf{P}}}
\nc{\sR}{{\mathsf{R}}}
\nc{\sS}{{\mathsf{S}}}
\nc{\sT}{{\mathsf{T}}}
\nc{\sU}{{\mathsf{U}}}
\nc{\sV}{{\mathsf{V}}}
\nc{\sZ}{{\mathsf{Z}}}

\newcommand{\up}[1]{\on{#1}}
\newcommand{\ul}[1]{\underline{#1}}
\newcommand{\ol}[1]{\overline{#1}}

\renewcommand{\epsilon}{\varepsilon}

\newcommand{\op}[0]{{\on{op}}}
\nc{\cat}{\on{cat}} %% cat for categorical.

\newcommand{\oo}[0]{\infty}

\newcommand{\uG}[0]{\on{G}}
\newcommand{\uH}[0]{\on H}
\newcommand{\uI}[0]{\on{I}}
\newcommand{\uK}[0]{\on{K}}
\newcommand{\uL}[0]{\on{L}}
\newcommand{\uM}[0]{\on{M}}
\newcommand{\HL}[0]{\GLK}%{\on{H}_{\uL}} old notation
\newcommand{\GLK}[0]{\on{G}_{\on{L/K}}}

\newcommand{\uKs}[0]{\overline{\on{K}}}

\newcommand{\coFib}[0]{\on{coFib}}

\newcommand{\Ker}[0]{\on{Ker}}
\nc{\act}{{\on{act}}}
\nc{\rev}{{\on{rev}}}
\nc{\env}{{\on{env}}}

\newcommand{\xto}[1]{\xrightarrow{#1}}
\nc{\hto}{\hookrightarrow}
\nc{\squigto}{\rightsquigarrow}
\nc{\longto}{\longrightarrow}
\nc{\xlongto}{\xlongrightarrow}
\nc{\tto}{\twoheadrightarrow}
\nc{\leftto}{\leftarrow}
\nc{\lto}{\leftto}
\nc{\longlto}{\longleftarrow}
\nc{\xlto}[1]{\xleftarrow{#1}}
\nc{\xlonglto}[1]{\xlongleftarrow{#1}}
\DeclareMathOperator{\id}{id} %identity

\newcommand{\Gal}[0]{{\scrG}al}
\newcommand{\IK}{{\uI_{\uK}}}
\newcommand{\IL}{{\uI_{\uL}}}

\nc{\kk}{\mathbbm{k}} %% ground field
\nc{\uk}{{\textup{k}}}
\newcommand{\Ql}[1]{\bbQ_{\ell #1}}
\newcommand{\BU}[0]{\rmB \bbU}

\nc{\sBenv}{\sB^{\env}}
\nc{\sCenv}{\sC^{\env}}

\nc{\Benv}{\CB^{\env}}
\nc{\Sym}{\on{Sym}}
\nc{\one}{{\mathbf{1}}}

\newcommand{\Spec}[0]{\on{Spec}}

\newcommand{\Sch}[0]{\on{Sch}}
\nc{\bareta}{\bar{\eta}}
\DeclareMathOperator{\Sm}{Sm}

\DeclareMathOperator{\Hom}{Hom} %Hom space
\newcommand{\pr}[0]{\on{pr}}

\nc{\Ind}[0]{\on{Ind}}

\newcommand{\HK}[0]{\on{HK}}

\nc{\LG}{\textup{LG}} %LG models

\newcommand{\dgCat}{\on{dgCat}}

\nc{\dgCats}{\on{dgCat}^{\up s}} %strict DGCat

\newcommand{\Mod}[0]{\on{Mod}}
\newcommand{\Fun}[0]{\on{Fun}}

\nc{\Sp}{\on{Sp}}
\nc{\shv}{\textup{Shv}} %sheaves
\newcommand{\cohs}[0]{\on{Coh}^{\on{s}}}

\DeclareMathOperator{\Coh}{\mathsf{D^b_{coh}}} %Coh
\DeclareMathOperator{\Perf}{\mathsf{D_{perf}}} %Perf
\DeclareMathOperator{\QCoh}{\mathsf{D_{qcoh}}} %QCoh
\DeclareMathOperator{\Sing}{\mathsf{D_{sg}}} %Singularity  cat
\DeclareMathOperator{\MF}{\on{MF}}

\newcommand{\SH}[0]{\ccS \ccH}
\newcommand{\SHnc}[0]{\ccS \ccH^{\on{nc}}}
\newcommand{\Shv}[0]{\on{Shv}_{\Ql{}}}

\newcommand{\Mv}[0]{\ccM^{\vee}}
\newcommand{\rl}[0]{\on{r}^{\ell}}

\nc{\hB}{\on{h}_{\sB}}
\nc{\hC}{\on{h}_{\sC}}

\nc{\restr}[2]{\left. #1 \right |_{#2}}

\nc{\wideprime}[1]{#1'}

\renewcommand{\setminus}{\smallsetminus}
\renewcommand{\sim}{\simeq}

\begin{document}

\title{Non-commutative nature of \texorpdfstring{$\ell$}{l}-adic vanishing cycles}
\address[D.~Beraldo]{Department of Mathematics, University College London, London WC1H 0AY, United Kingdom}

\email{d.beraldo@ucl.ac.uk}

\address[M.~ Pippi]{CNRS, Univ Angers, CNRS-UMR 6093, LAREMA, SFR MATHSTIC, F-49000 Angers, France}

\email{massimo.pippi@univ-angers.fr}

\author{Dario Beraldo \and Massimo Pippi}

\begin{abstract}
Let $p:X \to S$ be a flat (proper) and regular scheme of finite type over a strictly henselian discrete valuation ring. We prove that the singularity category of the special fiber with its natural two-periodic structure allows to recover the $\ell$-adic vanishing cohomology of $p$.

Along the way, we compute homotopy-invariant non-connective algebraic K-theory with compact support of certain embeddings $X_t\hto X_T$ in terms of the motivic realization of the dg-category of relatively perfect complexes.
\end{abstract}

\maketitle

\tableofcontents

\section{Introduction}\label{introduction}
\ssec{Posing the problem}

\sssec{}

It is well-known, and well documented in the existing literature, that differential graded (dg) categories of singularities are intimately related to vanishing cohomology. For instance, see \cite{dyc11,ef18,pr11,se13,brtv18}. 

\sssec{}

In particular, let $W$ be a complex smooth quasi-projective variety and $f: W \to \bbA^{1}_{\bbC}$ a regular map.
In \cite{ef18}, it is proven that the vanishing cohomology of $f$ \emph{together with its monodromy action} can be recovered as the periodic cyclic homology of the singularity category of $f^{-1}(0)$, with the extra datum given by a \emph{Getzler--Gauss--Manin connection}. 
The latter was introduced in \cite{kkp08} and written down explicitly in \cite{shk14}. 

\sssec{}

In this paper we deal with the $\ell$-adic analogue of the above phenomenon, where the extra datum of the Getzler--Gauss--Manin connection is replaced by a natural (left) module structure on the dg-category of singularities of the special fiber. 

\sssec{}

Let $S=\Spec(\scrO_{\uK})$ be the spectrum of an excellent strictly henselian discrete valuation ring, with closed point $i_S:s\to S$ and inertia group $\IK$. 
We assume that the residue field is perfect and fix a prime number $\ell$ different from the residue characteristic of $\scrO_{\uK}$.
Let $p:X\to S$ be a proper, flat and regular $S$-scheme.
The main result of \cite{brtv18} shows that it is possible to recover the homotopy $\IK$-fixed points of the $\ell$-adic vanishing cohomology of $p$ by means of derived and non-commutative algebraic geometry as follows.

\sssec{}

In \cite{brtv18}, Blanc--Robalo--To\"en--Vezzosi construct the \emph{$\ell$-adic realization of dg-categories} functor 
$$
 \rl_S:\dgCat_S \to \Mod_{\Ql{,S}(\beta)}\bigt{\Shv(S)},
$$
where the right hand side is the $\oo$-category of modules over $\Ql{,S}(\beta)=\bigoplus_{j\in \bbZ}\Ql{,S}(j)[2j]$ in the $\oo$-category of $\ell$-adic sheaves on $S$.
More details will be provided in Section \ref{ssec: motivic and l-adic realizations of dg-categories}.

\sssec{}

Consider the quotient dg-category
$$
  \Sing(X_s):=\frac{\Coh(X_s)}{\Perf(X_s)}.
$$
This is called the dg-category of singularities of the special fiber and it is naturally a module over the convolution monoidal dg-category $\Sing(G)$, where $G=s\times_S s$ is the derived self-intersection of the special point.

\sssec{}

The main theorem of \cite{brtv18} states that there is an equivalence
$$
  \rl_S\bigt{\Sing(X_s)} \simeq i_{S*}\uH^*_{\et} \Bigt{X_s,\Phi_p\bigt{\Ql{,X}(\beta)}}^{\IK}[-1],
$$
where $\Phi_p\bigt{\Ql{,X}(\beta)}$ denotes the $\ell$-adic sheaf of vanishing cyles of $p$ with $\Ql{,X}(\beta)=p^*\Ql{,S}(\beta)$ coefficients.
Moreover, this equivalence respects the natural actions of the algebra 
$$
\rl_S\bigt{\Sing(G)}\simeq i_{S*}\Ql{,s}(\beta)^{\IK}
$$
on both sides.

\medskip

It is natural to ask the following

\begin{quest}\label{main question}
Is it possible to recover the vanishing cohomology $\uH^*_{\et} \Bigt{X_s,\Phi_p\bigt{\Ql{,X}(\beta)}}$, with its natural continuous $\IK$-action, as the $\ell$-adic realization of a dg-category?
\end{quest}

\ssec{Our main results}

The goal of this paper is to provide an affirmative answer to the question above.

\sssec{}
Let $T=\Spec(\scrO_{\uL})\to S$ be a (necessarily totally ramified) extension of excellent strictly henselian discrete valuation rings. 
Let $\IL$ denote the absolute Galois group of the generic point $\Spec(\uL)$ of $T$. 
Let $i_{X_T}:X_t\hto X_T$ be the pullback of the closed immersion $X_s\hto X$ along $T\to S$. 
This morphism, being closed and quasi-smooth, induces a dg-functor 
$$
  i_{X_T*}:\Coh(X_t)\to \Coh(X_T)
$$ 
which preserves perfect complexes. In particular, it induces a dg-functor 
$$
  i_{X_T*}:\frac{\Coh(X_t)}{\Perf(X_t)}=\Sing(X_t)\to \Sing(X_T)=\frac{\Coh(X_T)}{\Perf(X_T)}
$$ 
at the level of the singularity dg-categories. 
Let us consider the dg-category of relative singularities of $X_t\hto X_T$ (\cite{bw12,ep15}),
$$
 \Sing (X_t\xto{i_{X_T}} X_T):= \Ker\bigt{i_{X_T*}:\Sing(X_t)\to \Sing(X_T) }.
$$
Our main theorem reads as follows.

\begin{mainthm}\label{main thm A}
Let $\HL$ denote the (finite) quotient of $\IK$ by $\IL$.
There is an equivalence
$$
  \rl_S\bigt{\Sing(X_t\xto{i_{X_T}} X_T)}\simeq i_{S*}\uH^*_{\et} \Bigt{X_s,\Phi_p\bigt{\Ql{,X}(\beta)}}^{\IL}[-1]
$$
of $i_{S*}\Ql{,s}^{\IL}(\beta)$-modules, compatible with the natural $\HL$-actions.
\end{mainthm}

\sssec{}

Given this result, it is then easy to answer to \questref{main question} as follows.

Let $\scrE$  be the filtered category of finite extensions of discrete valuation rings $T\to S$ as above. 
For two extensions $U\to T \to S$ as above, the pullback along $X_u\to X_t$ induces a dg-functor
$$
  \Sing(X_t\xto{i_{X_T}} X_T)\to \Sing(X_u\xto{i_{X_U}} X_U).
$$
This construction induces a diagram of dg-categories indexed by $\scrE$.
The actions of the finite quotients $\IK/\IL$ are compatible with this diagram and induce a continuous action of $\IK$ on the colimit
$$
  \ffS:= \varinjlim_{T\in \scrE}\Sing(X_t\xto{i_{X_T}} X_T).
$$

\begin{mainthm}\label{main thm B}
There is an equivalence
$$
  \rl_S(\ffS)\simeq i_{S*}\uH^*_{\et} \Bigt{X_s,\Phi_p\bigt{\Ql{,X}(\beta)}}[-1]
$$
of $i_{S*}\Ql{,s}(\beta)$-modules, compatible with the natural (continuous) $\IK$-actions.
\end{mainthm}

\begin{rmk}
    In Theorems \ref{main thm A} and \ref{main thm B}, it is not really necessary to assume that $p:X\to S$ is proper.
    In Appendix \ref{sec:prop hyp}, we explain how to remove this hypothesis.
\end{rmk}

\ssec{Strategy of the proof of \thmref{main thm A}}

\sssec{}

Observe that there is an equivalence of dg-categories
$$
  \Sing(X_t\xto{i_{X_T}} X_T)\simeq \frac{\Coh(X_t\xto{i_{X_T}} X_T)}{\Perf(X_t)},
$$
where $\Coh(X_t\xto{i_{X_T}} X_T)\subset \Coh(X_t)$ denotes the full subcategory spanned by objects $E \in \Coh(X_t)$ such that $i_{X_T*}(E)\in \Perf(X_T)$.
This is the dg-category of relatively perfect complexes we alluded to in the abstract.
In order to prove \thmref{main thm A}, one needs to compute the motivic realization of $\Coh(X_t\xto{i_{X_T}} X_T)$.

\sssec{}

We now notice that we have a localization sequence
$$
  \Coh(G_t\xto{a_1} t)\hto \Coh(G_t) \to \Sing(t)
$$
of dg-categories. Moreover, this is a localization sequence of left $\Coh(G)$-modules. 
Here, $a_1:G_t\to t$ is the pullback of $i_S:s\to S$ along $t \to S$ and the dg-functor $\Coh(G_t) \to \Sing(t)$ is induced by the pushforward along $a_1$ (notice that $a_1$ is proper and quasi-smooth).

\sssec{}

Since the dg-category $\Coh(X_s)^{\op}$ admits a right $\Coh(G)$-module structure, we can then apply the functor $\Coh(X_s)^{\op}\otimes_{\Coh(G)}-$ (i.e. the relative tensor product) and obtain the localization sequence
$$
\Coh(X_s)^{\op}\otimes_{\Coh(G)}\bigt{\Coh(G_t\xto{a_1} t) \hto \Coh(G_t) \to \Sing(t)}.
$$

\sssec{}

After computing these tensor products, we recognize that the rightmost dg-functor identifies with
$$
  \Coh(X_t)\to \Sing(X_T),
$$
the composition of $i_{X_T*}:\Coh(X_t)\to \Coh(X_T)$ with the quotient dg-functor $\Coh(X_T)\to \Sing(X_T)$.
As a consequence, we deduce that
$$
  \Coh(X_s)^{\op}\otimes_{\Coh(G)}\Coh(G_t\xto{a_1} t)\simeq \Coh(X_t\xto{i_{X_T}} X_T)
$$
and that
\begin{equation}\label{eqn: key loc seq}
  \Coh(X_t\xto{i_{X_T}} X_T)\hto \Coh(X_t)\to \Sing(X_T)
\end{equation}
is a localization sequence. 

\begin{rem}
This fact is nontrivial: even if $\Coh(X_t\xto{i_{X_T}} X_T)$ is by definition the kernel of $\Coh(X_t)\to \Sing(X_T)$, the equivalence
$$
  \frac{\Coh(X_t)}{\Coh(X_t\xto{i_{X_T}} X_T)}\xto{\sim}\Sing(X_T)
$$
is not obvious.
\end{rem}

\sssec{}

Now consider the \emph{motivic realization of dg-categories}, that is, the functor
$$
  \Mv_S:\dgCat_S\to \SH_S
$$
introduced in \cite{brtv18}, see Section \ref{ssec: motivic and l-adic realizations of dg-categories} for the details.
A fundamental property of $\Mv_S$ is that it sends localizations sequences to exact triangles. 
Using \eqref{eqn: key loc seq} as a key ingredient, we obtain that
\begin{equation}\label{eqnn:intro:Mv of Coh-rel}
      \Mv_S\bigt{\Coh(X_t\xto{i_{X_T}} X_T)}\simeq \Mv_S\bigt{\Perf(X_T)_{X_t}}.
\end{equation}
where $\Perf(X_T)_{X_t}$ denotes the dg-category of perfect complexes on $X_T$ with set-theoretic support contained in $X_t$.

\begin{rem}
The equivalence \eqref{eqnn:intro:Mv of Coh-rel} could be regarded as a form of \emph{d\'evissage for homotopy-invariant non-connective algebraic K-theory} and seems to be a new result interesting on its own, see \thmref{devissage-like thm}.
\end{rem}

\sssec{} 

Once the above computation of $\Mv_S\bigt{\Coh(X_t\xto{i_{X_T}} X_T)}$ is settled, we can proceed similarly to \cite{brtv18} and conclude the proof of \thmref{main thm A}.

\begin{rem}
This work is the second in a series of three papers whose goal is to prove the Deligne--Milnor conjecture following the vision of \TV{}. The first and third paper of the series are \cite{bp22} and \cite{bp24}, respectively.
\end{rem}

\ssec*{Acknowledgements}

We are grateful to Mauro Porta, Marco Robalo, Bertrand To\"en and Gabriele Vezzosi for several inspiring discussions. We thank an anonymous referee for several suggestions that contributed to improve the paper.
    
The first version of the paper was written while MP was supported by the collaborative research center SFB 1085 \emph{Higher Invariants - Interactions between Arithmetic Geometry and Global Analysis} funded by the Deutsche Forschungsgemeinschaft.
This project has been partially supported by the PEPS JCJC 2024.

\section{Preliminaries}

\ssec{Notations} 

We fix here some notations that we will adopt in the main body of the paper.

\sssec{}

Let $\scrO_{\uK}$ be a complete\footnote{In the introduction, we only assumed $\scrO_{\uK}$ to be excellent and strictly henselian. This further assumption on $\scrO_{\uK}$ is harmless in view of \cite[Expos\'e XIII, Proposition 2.1.12]{sga7ii}. 
} strict discrete valuation ring and $\uK \supseteq \scrO_{\uK}$ its fraction field. 
We assume that the residue field is perfect.

We fix once and for all a uniformizing element $\pi_{\uK} \in \scrO_{\uK}$ and denote by $\uk=\scrO_{\uK}/(\pi_{\uK})$ the (algebraically closed) residue field. 
We also fix a separable closure $\uKs$ of $\uK$ and denote by $\IK=\Gal(\uKs/\uK)$ the absolute Galois group of $\uK$, which coincides with the inertia group in this case.
Moreover, let $S$ (resp. $s$, $\eta$, $\bareta$) be the spectrum of $\scrO_{\uK}$ (resp. $\uk$, $\uK$, $\uKs$):
$$
  s\xlongto{i_S} S \xlonglto{j_S} \eta \longlto \bareta.
$$

\sssec{}\label{notat: extensions dvr}

Let $\uK \subseteq \uL$ be a finite Galois extension (viewed inside $\uKs$), which is necessarily totally ramified, and assume that the ring of integers $\scrO_{\uL}$ of $\uL$ is still a (strictly henselian) trait. 
In this case, for a fixed uniformizing element $\pi_{\uL}\in \scrO_{\uL}$, there is a unit $u\in \scrO_{\uL}^{\times}$ such that
$$
  \pi_{\uK}=u\cdot \pi_{\uL}^\ue,
$$
where $\ue=[\uL:\uK]$ is the degree of the extension (which agrees with the ramification degree in this case).

\sssec{}
Denote by $\IL=\Gal(\uKs/\uL)$ the absolute Galois group of $\uL$: this is an open normal subgroup of $\IK$. Let $\HL\simeq \Gal(\uL/\uK)$ be the (finite) quotient group $\IK/\IL$. Set $T:=\Spec(\scrO_{\uL})$ and denote by $t$ (resp. $\eta_{\uL}$, $\bareta_{\uL}$) the pullback of $s$ (resp. $\eta$, $\bareta$) along $T\to S$. We thus have Cartesian squares
\begin{equation*}
    \begin{tikzcd}[column sep= huge, row sep=large]
        t \arrow[r,"i_T"] \arrow[d] & \arrow[d]  T & \arrow[l,swap,"j_T"]  \eta_{\uL}\arrow[d] &  \arrow[l]  \bareta_{\uL} \arrow[d]
        \\
        s \arrow[r,"i_S"]  & S  & \arrow[l,swap,"j_S"] \eta & \arrow[l] \bareta. 
    \end{tikzcd}
\end{equation*}
Notice that $t\simeq \Spec\bigt{\scrO_{\uL}/(\pi_{\uK})}$ is a nil-thickening of $\Spec(\uk)$, while $\eta_{\uL}=\Spec(\uL)$ and  $\bareta_{\uL}\simeq \HL \times \bareta$.

\sssec{}

We denote by $G_t\xto{a_1} t$ the pullback of $G\hto s$ (i.e. the first projection $s\times_S s\to s$) along $t\to s$. In fact, it is easy to see that we can write $G_t$ as the (derived) pullback $t\times_T t$ and that $a_1$ agrees with the first projection:
\begin{equation*}
    \begin{tikzcd}[column sep= huge, row sep=large]
        t\times_Tt \arrow[r,"a_1"] \arrow[d,swap,"a_2"] & t \arrow[d,"i_T"]\\
        t \arrow[r, "i_T"] & T.
    \end{tikzcd}
\end{equation*}
Under the equivalence $t\times_Tt\simeq s\times_Ss\times_ST$, the map $a_1$ corresponds to the projection
$$
  \hat{\pr}_{13}:
  s\times_Ss\times_ST
  \to
  s\times_ST
$$
onto the first and third component.

\sssec{}

Throughout this paper, we will consider a proper and flat $S$-scheme $p:X\to S$, which is moreover assumed to be regular (and generically smooth). We will denote by $p_s:X_s\to s$ (resp. $p_{\uK} : X_{\uK} \to \eta$, $p_{\uKs} :X_{\uKs} \to \bareta$) the pullback of $p:X\to S$ along $s\hto S$ (resp. $\eta \hto S$, $\bareta \to S$), so that we obtain a diagram
$$
  X_s \xlongto{i_X} X \xlonglto{j_X} X_{\uK} \longlto X_{\uKs}. 
$$

\sssec{}

Similarly, we denote by $p_T:X_T\to T$ (resp. $p_t:X_t\to t$, $p_{\uL}:X_{\uL} \to \eta_{\uL}$) the pullback of $p:X\to S$ along $T\to S$ (resp. $t\to S$, $\eta_T\to S$), and get the open-closed decomposition
$$
  X_t\xlongto{i_{X_T}} X_T \xlonglto{j_{X_T}} X_{\uL}.
$$

\ssec{Higher categories} 

\sssec{}

We will freely use the theory of higher categories, see \cite{lu17, lu09}.
All functors are implicitly derived. 
Morphisms between $\oo$-categories are simply called \virg{functors}, instead of the more precise \virg{$\infty$-functors}.

\sssec{} 

We work in the framework of dg-categories \emph{up to Morita equivalences}. 
We refer to \cite{ke06, to07, to11} for exhaustive accounts.

\sssec{}

Let $\dgCats_S$ denote the (ordinary) category of small $\scrO_{\uK}$-linear dg-categories (i.e. categories enriched in cochain complexes of $\scrO_{\uK}$-modules).
A dg-functor is then just a functor compatible with these enrichments.

\sssec{}

For a dg-category $\sT$, its homotopy category is the ($\scrO_{\uK}$-linear) category $\uh \sT$ with the same objects of $\sT$ and such that
$$
  \Hom_{\uh \sT}(x,y)=\uH^0\bigt{\Hom_{\sT}(x,y)}
$$
for any objects $x,y$.

\sssec{}

Among all dg-functors, we consider the collection $\on{W}_{\on{Mor}}$ of \emph{Morita equivalences}, that is, those dg-functors $F:\sT\to \sU$ such that $F$ induces a quasi-isomorphism 
$$
  \Hom_{\sT}(x,y)\to \Hom_{\sU}\bigt{F(x),F(y)}
$$
for all pairs of objects $x,y \in \sT$ and the image of $F$ generates the Karoubi completion $\widehat{\sU}_c$ of $\sU$ (recall that $\sU \subseteq \widehat{\sU}_c$) under cones, shifts and retracts.
Then we consider the $\infty$-localization of $\dgCats_S$ along $\on{W}_{\on{Mor}}$:
$$
  \dgCat_S:=\dgCats_S[\on{W}_{\on{Mor}}^{-1}].
$$

\begin{rem}
  This $\oo$-localization has a model. 
  Indeed, in \cite{tab05}, G.~Tabuada exhibits a model category structure on $\dgCats_S$ where weak-equivalences are \emph{quasi-equivalences}, i.e. dg-functors inducing quasi-isomorphisms on the hom complexes and which induce equivalences on the homotopy categories. 
  Every quasi-equivalence is a Morita equivalence and one can take the associated Bousfield localization, which is a model category whose associated $\oo$-category is equivalent to $\dgCat_S$.
\end{rem}

\sssec{}

In \cite{to07}, B.~To\"en showed that there is a well behaved theory of dg-localizations. 
In other words, for every $\sT \in \dgCat_S$ and every (saturated) collection of morphisms $\sW \subset \sT$, there exists a dg-category $\sT[\sW^{-1}]\in \dgCat_S$ endowed with a dg-functor $\sT \to \sT[\sW^{-1}]$ which has the following universal property: it induces a fully faithful embedding of functor $\oo$-categories 
$$
  \Fun_{\dgCat_S}\bigt{\sT[\sW^{-1}],\sU}\to \Fun_{\dgCat_S}\bigt{\sT,\sU}
$$
for every $\sU\in \dgCat_S$, whose essential image consists of dg-functors $\sT \to \sU$ mapping every morphism in $\sW$ to an equivalence.

\sssec{}

There is also a theory of dg-quotients: for a sub dg-category $\sU \subset \sT$, the dg-quotient $\sT/\sU$ is the dg-localization of $\sT$ along those morphisms $x\to y$ in $\sT$ whose fiber belongs to $\sU$.
More generally, for a dg-functor $F:\sU \to \sT$, the dg-quotient $\sT/\sU$ is defined as the dg-quotient of $\sT$ by the full sub dg-category spanned by the essential image of $F$.

\sssec{}

Of major relevance for the purposes of this paper is the notion of \emph{localization sequence} in $\dgCat_S$. We say that a diagram
$$
  \sT_1\to \sT_2\to \sT_3
$$
in $\dgCat_S$ is a localization sequence if the composition is homotopic to $0$, the induced dg-functor
$$
  \sT_2/\sT_1\to \sT_3
$$
is a Morita equivalence and $\sT_1$ is the kernel of $\sT_2\to \sT_3$.

\ssec{Motivic and \texorpdfstring{$\ell$}{l}-adic realizations of dg-categories}\label{ssec: motivic and l-adic realizations of dg-categories} 
We recall  here some of the main constructions of \cite{brtv18}.

\sssec{}

Let $\SH_S$ denote the stable homotopy category of $S$-schemes introduced by F.~Morel and V.~Voevodsky in \cite{mv99} (or rather its $\oo$-categorical version, see \cite{ro15}). 
This is a stable symmetric monoidal presentable $\oo$-category endowed with a symmetric monoidal functor 
$$
  \Sigma^{\oo}_+:\Sm_S \to \SH_S
$$
which enjoys the following universal property (see \cite{ro15}). Suppose that $\ccC$ is a stable presentable symmetric monoidal $\oo$-category endowed with a symmetric monoidal functor $F:\Sm_S\to \ccC$ such that

\begin{itemize}
    \item $F$ satisfies Nisnevich descent (i.e. it sends Nisnevich squares in $\Sch_S$ to pullbacks in $\ccC$),
    \item the canonical map $\bbA^{1}_{S}\to S$ is mapped to an equivalence in $\ccC$,
    \item the fiber of $F\bigt{S\xto{\oo} \bbP^{1}_{S}}$ is an invertible object in $\ccC$;
\end{itemize}

then $F$ must factor (essentially uniquely) through $\Sigma^{\oo}_+$. 

\sssec{}

Among motivic spectra (that is, objects in $\SH_S$) there is $\BU_S$, the spectrum which represents homotopy-invariant non-connective algebraic K-theory. 
This object enjoys the \emph{algebraic Bott periodicity}, i.e. there is a canonical equivalence
$$
  \BU_S\simeq \BU_S(1)[2].
$$

\sssec{}

In \cite{ro15}, M.~Robalo constructs a non-commutative analogue of $\SH_S$ (see also \cite{ct11,ct12} for an alternative construction). 
This is a stable symmetric monoidal presentable $\oo$-category $\SHnc_S$ equipped with a symmetric monoidal functor
$$
  \iota:\dgCat_S^{\on{ft},\op} \to \SHnc_S
$$
from the opposite $\oo$-category of dg-categories of finite type (see \cite{tv07}) which enjoys the analogue universal property of $\SH_S$: for every symmetric monoidal functor $F:\dgCat_S^{\on{ft},\op}\to \ccC$, where $\ccC$ is a stable presentable symmetric monoidal $\oo$-category, such that
\begin{itemize}
    \item $F$ sends Nisnevich squares of dg-categories (see \cite{ro15}) to pullbacks in $\ccC$,
    \item the morphism $\Perf(\bbA^{1}_{S})\to \Perf(S)$ in $\dgCat_S^{\on{ft},\op}$ induced by pullback along the projection map is mapped to an equivalence in $\ccC$,
    \item the fiber of $F\bigt{\Perf(S)\xto{\oo^*}\Perf(\bbP^{1}_{S})}$ is an invertible object in $\ccC$,
\end{itemize}
then $F$ factors (essentially uniquely) through $\iota$.

\sssec{}

The composition
$$
  \Sm_S\xto{\Perf} \dgCat^{\on{ft},\op}_S\xto{\iota} \SHnc_S
$$
is symmetric monoidal and enjoys all the properties listed above. 
By the universal property of $\SH_S$, we thus obtain a functor
$$
  \ccR_{\on{pe}}: \SH_S\to \SHnc_S,
$$
called the \emph{perfect realization}. 
This is a (symmetric monoidal) colimit preserving functor between presentable stable $\oo$-categories, thus it admits a (lax-monoidal) right adjoint 
$$
  \ccM_S:\SHnc_S\to \SH_S.
$$
As proved in \cite{ro15}, this functor maps the unit object $\one^{\on{nc}}_S$ of $\SHnc_S$ to $\BU_S$.

\sssec{}

In \cite{brtv18}, the following \virg{dual} version of $\ccM_S$ is considered:
$$
  \Mv_S:\dgCat^{\on{ft}}_S\xto{\iota^\op} \SH_S^{\on{nc},\op} \xto{\ul \Hom_{\SHnc_S}(-,\one^{\on{nc}}_S)} \SHnc_S \xto{\ccM_S}\SH_S.
$$
Since $\SH_S$ is presentable and $\Ind\bigt{\dgCat^{\on{ft}}_S}\simeq \dgCat_S$ (see \cite{tv07}), 
%%%
we can extend this (lax-monoidal) functor to $\dgCat_S$:
$$
  \Mv_S:\dgCat_S\to \SH_S.
$$
This is called the \emph{motivic realization of dg-categories}. As it is lax-monoidal, we actually get a functor
$$
  \Mv_S:\dgCat_S\to \Mod_{\BU_S}(\SH_S).
$$
For a dg-category $\sT$, the motivic spectrum underlying $\Mv_S(\sT)$ is a functor $\Sm_S^{\op}\to \Sp$ (here $\Sp$ denotes the stable presentable symmetric monoidal $\oo$-category of spectra) defined on objects by the assignment
$$
  Y\mapsto \HK \bigt{\sT\otimes_{\Perf(S)}\Perf(Y)},
$$
where $\HK$ denotes homotopy-invariant non-connective algebraic $\uK$-theory.

\sssec{}

The motivic realization of dg-categories enjoys the following properties:

\begin{itemize}
  \item it preserves filtered colimits;
  \item for every qcqs $S$-scheme $q:Y\to S$ of finite type, $\Mv_S\bigt{\Perf(Y)}\simeq q_*\BU_Y$;
  \item it sends localization sequences in $\dgCat_S$ to fiber-cofiber sequences in $\SH_S$.
\end{itemize}

\sssec{}

Let $\ell$ be a prime number invertible in $\scrO_{\uK}$. The authors of \cite{brtv18} considered also the $\ell$-adic realization
$$
  \ccR^{\ell}_S:\SH_S \xto{-\otimes \uH\!\bbQ} \Mod_{\uH\!\bbQ}(\SH_S) \to \Shv(S),
$$
where $\uH\!\bbQ$ is the spectrum of rational singular cohomology.
The second functor is constructed in \emph{loc. cit.} (based on the rigidity theorems due to Ayoub and Cisinski--D\'eglise, see \cite{ay14,cd16}). 
It is a symmetric monoidal functor with values in the $\oo$-category of ind-constructible $\ell$-adic sheaves.
It follows from results of J.~Riou (see \cite{ri10}) that $\ccR^{\ell}_S(\BU_S)\simeq \Ql{,S}(\beta)=\bigoplus_{j\in \bbZ}\Ql{}(j)[2j]$.

\sssec{}

The composition
$$
  \rl_S:\dgCat_S \xto{\Mv_S} \Mod_{\BU_S}(\SH_S) \xto{\ccR^{\ell}_S} \Mod_{\Ql{,S}(\beta)}=\Mod_{\Ql{,S}(\beta)}(\shv_{\Ql{}}(S))
$$
is a lax-monoidal functor which enjoys the same properties of $\Mv_S$ (\emph{mutatis mutandis}) and it is called the \emph{$\ell$-adic realization of dg-categories}.

\ssec{Some dg-categories of interest} We will be interested in some very specific dg-categories. 

\sssec{}

Let $Y$ denote a (possibly derived) scheme of finite type over $S$. One associates to it its dg-category of quasi coherent complexes $\QCoh(Y)$. We will need to consider the following two sub-dg-categories:

\begin{itemize}
    \item the full subcategory $\Coh(Y)$, spanned by those complexes with coherent and bounded cohomology sheaves;
    \item the full subcategory $\Perf(Y)$ of perfect complexes.
\end{itemize}

Under the mild hypothesis that the structure sheaf $\ccO_Y$ is bounded, we have a fully faithful embedding $\Perf(Y)\subseteq \Coh(Y)$. 
In this case, the \emph{dg-category of (absolute) singularities of $Y$} is defined as the dg-quotient
$$
  \Sing(Y):=\frac{\Coh(Y)}{\Perf(Y)}.
$$

\sssec{}

We will also need to consider the following. 
Let $j:Z\hto Y$ be a quasi-smooth closed embedding of (derived) schemes of finite type over $S$. 
Then the pushforward $j_*:\QCoh(Z)\to \QCoh(Y)$ induces dg-functors $j_*:\Coh(Z)\to \Coh(Y)$ and $j_*:\Perf(Z)\to \Perf(Y)$.
Therefore, it induces a dg-functor
$$
  j_*:\Sing(Z)\to \Sing(Y).
$$
The \emph{dg-category of relative singularities of $Z\hto Y$} is defined as the kernel of this dg-functor:
$$
  \Sing(Z\xto{j} Y):= \Ker \bigt{j_*:\Sing(Z)\to \Sing(Y)}.
$$

\sssec{}

This dg-category also admits an alternative description.
Let $\Coh(Z\xto{j} Y)$ denote the kernel of the dg-functor
$$
  \Coh(Z)\xlongto{j_*} \Coh(Y) \tto \Sing(Y).
$$
This is the full subcategory of $\Coh(Z)$ spanned by those complexes $E\in \Coh(Z)$ whose image along $j_*$ is a perfect complex of $Y$. 
Since $j_*$ preserves perfect complexes, all perfect complexes over $Z$ lie in this subcategory: $\Perf(Z)\subseteq \Coh(Z\xto{j} Y)$. 
Thus, there is an equivalence
$$
  \frac{\Coh(Z\xto{j} Y)}{\Perf(Z)}\xto{\sim} \Sing(Z\xto{j} Y).
$$

\ssec{The monoidal dg-categories \texorpdfstring{$\sB^+$}{sB+} and \texorpdfstring{$\sB$}{sB}}

Following \cite{tv22}, we now introduce two important monoidal dg-categories. 

\sssec{}

Consider the derived fiber product $G:=s\times_S s$, i.e. the spectrum of the simplicial Koszul algebra $\uK\bigt{\scrO_{\uK},(\pi_{\uK},\pi_{\uK})}$. This is a derived groupoid scheme over $s$ (i.e. $\uK\bigt{\scrO_{\uK},(\pi_{\uK},\pi_{\uK})}$ is a Hopf algebroid).
The composition $G\times_s G\to G$ corresponds to the projection onto the first and third factor under the equivalence $G\times_s G\simeq s\times_S s\times_S s$, while the unit corresponds to the canonical morphism $u:s\to G$.

\sssec{}

This derived groupoid structure induces a monoidal convolution $\odot$ product on $\sB^+:=\Coh(G)$. 
Roughly, this is defined as the dg-functor
$$
 - \odot - :\sB^+ \otimes \sB^+ \to \sB^+
$$
$$
 (M,N)\to \pr_{13*}\bigt{\pr_{12}^*M\otimes \pr_{23}^*N},
$$
where $\pr_{ij}:G\times_sG\simeq  s\times_S s\times_S s\to G$ denotes the projection onto the $i^{\on{th}}$ and $j^{\on{th}}$ factors (which is a proper quasi-smooth map). 
The unit of this convolution product is $u_*\ccO_s$; in other words, it is $\uk$ with the obvious $\ccO_G$-module structure.

\begin{rem}
  Beware that this convolution product is associative and unital (up to coherent homotopy), but not commutative in general. In other words, $\sB^+$ is just an $\bbE_1$-algebra in $\dgCat_S$.
\end{rem}

\sssec{}

The above convolution product is compatible with perfect complexes, i.e. $M\odot N \in \Perf(G)$ as soon as $M$ or $N$ lies in $\Perf(G)$. 
Therefore, $\odot$ induces a similarly defined convolution product on the dg-category of singularities $\sB:=\Sing(G)$.

\sssec{} 

We will denote the $\oo$-category of left (resp. right) $\sB^+$-modules by $\dgCat_{\sB^+}$ (resp. $\dgCat^{\sB^+}$). An analogous notation will be employed for left (resp. right) $\sB$-modules.

\section{A useful localization sequence}

The goal of this section is to construct a localization sequence of dg-categories
$$
  \Coh(G_t\xto{a_1} t) \hto \Coh(G_t) \to \Sing(t),
$$
and then show that it is $\sB^+$-linear for the natural $\sB^+$-module structures on all terms.

\ssec{Explicit models} 

\sssec{}

Recall from \cite{pi22b} that we have the following explicit models for $\Coh(G_t)$ and $\Coh(t)$. 
The simplicial algebra $\uK \bigt{\scrO_{\uL},(\pi_{\uK}, \pi_{\uK})}$ corresponds, under the Dold-Kan equivalence, to the dg-algebra (also denoted $\uK \bigt{\scrO_L,(\pi_{\uK}, \pi_{\uK})}$, by abuse of notation)
$$
  \scrO_{\uL} \cdot h_1 h_2 \longto \scrO_{\uL} \cdot h_1 \oplus \scrO_{\uL} \cdot h_2 \longto \scrO_{\uL}
$$
placed in (cohomological) degrees $[-2,0]$ and with differential characterized by the requirement that $h_1,h_2\mapsto \pi_{\uK}$. 
Also, the variables $h_1,h_2$ anticommute and square to zero.

\sssec{}

Similarly, the simplicial algebra $\uK (\scrO_{\uL},\pi_{\uK})$ corresponds, under the Dold-Kan equivalence, to the dg-algebra (also denoted $\uK (\scrO_{\uL},\pi_{\uK})$)
$$
  \scrO_{\uL} \cdot h \longto \scrO_{\uL}
$$
placed in (cohomological) degrees $[-1,0]$ and with differential characterized by $h\mapsto \pi_{\uK}$. 
The variable $h$ squares to zero.

\sssec{}

For $i=1,2$, the morphism
$$
  a_i:t\times_T t \simeq \Spec \Bigt{\uK \bigt{\scrO_L,(\pi_{\uK}, \pi_{\uK})}} \to \Spec\bigt{\uK (\scrO_{\uL},\pi_{\uK})}\simeq t
$$
corresponds to the morphism of simplicial algebras
$$
  \uK (\scrO_{\uL},\pi_{\uK}) \to \uK \bigt{\scrO_L,(\pi_{\uK}, \pi_{\uK})}
$$
uniquely determined by $h \mapsto h_i$.

\sssec{} 

Let $\cohs\bigt{\scrO_{\uL},(\pi_{\uK},\pi_{\uK})}$ denote the strict $\scrO_{\uK}$-dg-category of dg-modules over $\uK \bigt{\scrO_{\uL},(\pi_{\uK}, \pi_{\uK})}$ with strictly perfect underlying $\scrO_{\uL}$-dg-modules. 
More explicitly, it is defined as follows:
  \begin{itemize}
      \item the objects of $\cohs\bigt{\scrO_{\uL},(\pi_{\uK},\pi_{\uK})}$ are tuplets $(E,d,\{h_1,h_2\})$, where $(E,d)$ is a strictly perfect cochain complex of $\scrO_{\uL}$-modules (i.e. degreewise projective of finite type and strictly bounded) and each $h_i$ is a $\scrO_{\uL}$-linear morphism 
      $$
        h_i:E \to E[-1]
      $$
      of degree $-1$. These data are subject to the following requirements:
   \begin{enumerate}
       \item   $h_i\circ h_i=0$, for $i=1,2$; 
       \item $[h_1,h_2]=0$;
       \item  $[d,h_i]=\pi_{\uK} \cdot \id_E$.
   \end{enumerate}
        
 \medskip
 
      \item for two such objects 
      $$\bbE=(E,d_E,\{h_{i,E}\}_{i=1,2})$$ 
      $$\bbF=(F,d_F,\{h_{i,F}\}_{i=1,2})$$
      and for each $n\in \bbZ$, the $\scrO_{\uK}$-module of degree $n$ morphisms $\Hom^n(\bbE,\bbF)$ is the submodule of 
      $$
        \bigoplus_{j\in \bbZ} \Hom_{\scrO_{\uL}}(E^j,F^{j+n}) 
      $$
      (which is isomorphic to $\displaystyle \prod_{j\in \bbZ} \Hom_{\scrO_{\uL}}(E^j,F^{j+n})$ as the complexes are strictly bounded)
      spanned by those elements $\{ \phi^j:E^j\to F^{j+n}\}_{j\in \bbZ}$ verifying the equations
      $$
        \phi^j\circ h_{i,E}^{j+1}= h_{i,F}^{j+n+1}\circ \phi^{j+1}, \hspace{0.5cm} i=1,2.
      $$
      As usual, these modules form a cochain complex by considering the differential
      $$
        \Hom^n(\bbE,\bbF)\to \Hom^{n+1}(\bbE,\bbF)
      $$
      $$
        \{ \phi^j:E^j\to F^{j+n}\}_{j\in \bbZ}\mapsto \{ d_F^{j+n}\circ \phi^j+(-1)^{n+1}\phi^{j+1}\circ d_E^j:E^j\to F^{j+n+1}\}_{j\in \bbZ}.
      $$
  \end{itemize}

\begin{rem}
Since strictly perfect cochain complexes of $\scrO_{\uL}$-modules are degreewise projective of finite rank and strictly bounded, the $\scrO_{\uL}$-module $\bigoplus_{j\in \bbZ} \Hom_{\scrO_{\uL}}(E^j,F^{j+n})$ is projective of finite rank for each $n\in \bbZ$. 
  As $\scrO_{\uL}$ is a principal ideal domain, this means that $\bigoplus_{j\in \bbZ} \Hom_{\scrO_{\uL}}(E^j,F^{j+n})$ is free of finite rank for each $n\in \bbZ$. 
  Therefore, $\Hom^n(\bbE,\bbF)\subseteq \bigoplus_{j\in \bbZ} \Hom_{\scrO_{\uL}}(E^j,F^{j+n})$ is free of finite rank for each $n\in \bbZ$ as well.
  Since $\scrO_{\uL}$ is a (faithfully) flat $\scrO_{\uK}$-algebra, $\cohs\bigt{\scrO_{\uL},(\pi_{\uK},\pi_{\uK})}$ is a locally flat $\scrO_{\uK}$-dg-category.
\end{rem}

\sssec{}

We have an analogous model for $\Coh(t)$.
Let $\cohs (\scrO_{\uL},\pi_{\uK})$ denote the strict $\scrO_K$-dg-category of $\uK (\scrO_{\uL},\pi_{\uK})$ dg-modules with strictly perfect underlying $\scrO_{\uL}$-dg-module. 
This dg-category can be described explicitly as well:
\begin{itemize}
    \item the objects of $\cohs (\scrO_{\uL},\pi_{\uK})$ are tuplets $(E,d,h)$, where $(E,d)$ is a strictly perfect complex of $\scrO_{\uL}$-modules (i.e. degreewise projective of finite rank and strictly bounded) and
    $$
      h:E \to E[-1]
    $$
    is a $\scrO_{\uL}$-linear morphism of degree $-1$ such that
    \begin{enumerate}
        \item $h^2=0$;
        \item $[d,h]=\pi_{\uK} \cdot \id_E$.
    \end{enumerate}
    \item for two such objects 
    $$\bbE=(E,d_E,h_E)$$
    $$\bbF=(F,d_F,h_F)$$ 
    and for each $n\in \bbZ$, the $\scrO_{\uK}$-module $\Hom^n(\bbE,\bbF)$ is the submodule of 
    $$
      \bigoplus_{j\in \bbZ} \Hom_{\scrO_{\uL}}(E^j,F^{j+n})
    $$
    (which is isomorphic to $\displaystyle \prod_{j\in \bbZ} \Hom_{\scrO_{\uL}}(E^j,F^{j+n})$ as the complexes are strictly bounded)
    spanned by those elements $\{\phi^j:E^j\to F^{j+n}\}_{j\in \bbZ}$ such that
    $$
      \phi^j\circ h_E^{j+1}=h_F^{j+n+1}\circ \phi^{j+1}.
    $$
    The $\scrO_{\uK}$-module $\bigoplus_{n\in \bbZ}\Hom^n(\bbE,\bbF)$ is equipped with the same differential as above.
\end{itemize}

\begin{rem}
  Just as in the previous remark, $\cohs (\scrO_{\uL},\pi_{\uK})$ is a locally flat $\scrO_{\uK}$-dg-category.
\end{rem}

\sssec{}

According to \cite{brtv18, pi22b}, the strict dg-categories $\cohs \bigt{\scrO_{\uL},(\pi_{\uK},\pi_{\uK})}$ and $\cohs (\scrO_{\uL},\pi_{\uK})$ are strict models for $\Coh(G_t)$ and $\Coh(t)$:

\begin{lem}
  Let $\on{W}_{\on{qi}}$ denote the class of quasi-isomorphisms in both the dg-categories $\cohs \bigt{\scrO_{\uL},(\pi_{\uK},\pi_{\uK})}$ and $\cohs (\scrO_{\uL},\pi_{\uK})$. 
  Then
  $$
    \cohs \bigt{\scrO_{\uL},(\pi_{\uK},\pi_{\uK})}[\on{W}_{\on{qi}}^{-1}]_{\on{dg}}\simeq \Coh(G_t),
  $$
  $$
    \cohs (\scrO_{\uL},\pi_{\uK})[\on{W}_{\on{qi}}^{-1}]_{\on{dg}}\simeq \Coh(t).
  $$
Here, on the left hand side we consider the localization of dg-categories introduced by B.~To\"en (\cite{to07}).
\end{lem}

\sssec{}

Using these strict models, it is easy to give strict models of the dg-functors

$$
  a_i^*: \Coh(t)\to \Coh(G_t), \hspace{0.5cm} i=1,2
$$
and 

$$
  a_{i*}:\Coh(G_t)\to \Coh(t), \hspace{0.5cm} i=1,2.
$$
Indeed, the pushforward along $a_i$ ($i=1,2$) corresponds to the dg-functor

$$
  \cohs \bigt{\scrO_{\uL},(\pi_{\uK},\pi_{\uK})} \to \cohs (\scrO_{\uL},\pi_{\uK})
$$
defined on objects by

$$
  (E,d,\{h_s\}_{s=1,2})\to (E,d,h_i)
$$
and on morphisms by the inclusion
\begin{eqnarray}
\nonumber
  \Hom\bigt{(E,d_E,\{h_{E,s}\}_{s=1,2}),(F,d_F,\{h_{F,s}\}_{s=1,2})} & \subseteq & \Hom\bigt{(E,d_E,\{h_{E,i}\}),(F,d_F,\{h_{F,i}\}} 
  \\
  \nonumber
  & \subseteq &  \bigoplus_{j,n\in \bbZ}\Hom_{\scrO_{\uL}}(E^j,F^{j+n}). 
\end{eqnarray}
This is obviously compatible with the differentials, with the identities and with the composition of morphisms. Moreover, it obviously preserves quasi-isomorphisms and its dg-localization along $\on{W}_{\on{qi}}$ is equivalent to
$$
  a_{i*}:\Coh(G_t)\to \Coh(t), \hspace{0.5cm} i=1,2.
$$

\sssec{}

The pullback along $a_i$ ($i=1,2$) can be \virg{strictified} as well. 
Consider the dg-functor
$$
  \cohs (\scrO_{\uL},\pi_{\uK})\to \cohs \bigt{\scrO_{\uL},(\pi_{\uK},\pi_{\uK})}
$$
defined by sending an object $(E,d,h)$ to $E \oplus E[1]$ with differential
$$
  \begin{bmatrix}
    d & \pi_{\uK} \cdot \id_E 
    \\
    0 & -d
  \end{bmatrix}: E \oplus E[1]\to E[1] \oplus E[2]
$$
and where $h_i \in \uK \bigt{\scrO_{\uL},(\pi_{\uK},\pi_{\uK})}$ acts via
$$
  \begin{bmatrix}
    h & 0
    \\
    0 & -h
  \end{bmatrix}: E \oplus E[1]\to E[-1] \oplus E,
$$
while $h_j\in \uK \bigt{\scrO_{\uL},(\pi_{\uK},\pi_{\uK})}$ ($j\in \{1,2\}\setminus \{i\}$) acts via
$$
  \begin{bmatrix}
    0 & 0
    \\
    \id_E & 0
  \end{bmatrix}: E \oplus E[1]\to E[-1] \oplus E.
$$
It is defined on morphisms as
$$
  \phi:E\to F \mapsto 
  \begin{bmatrix}
    \phi & 0
    \\
    0 & \phi[1]
  \end{bmatrix}
  : E \oplus E[1] \to F \oplus F[1].
$$
It is straightforward to verify that this is compatible with the differentials, with the identities and with the composition.
This dg-functor also preserves quasi-isomorphisms and its localization along $\on{W}_{\on{qi}}$ is equivalent to
$$
  a_i^*:\Coh(t)\to \Coh(G_t), \hspace{0.5cm} i=1,2.
$$

\begin{rem}
The two pairs of dg-functors
  $$
    a_{i}^*:
    \cohs (\scrO_{\uL},\pi_{\uK})
    \rightleftarrows 
    \cohs \bigt{\scrO_{\uL},(\pi_{\uK},\pi_{\uK})}
    :a_{i*}
  $$
are adjunctions.
\end{rem}

\begin{rem}
  The dg-functor $a_2^* : \Coh(t)\to \Coh(G_t)$ factors through the full embedding $\Coh(G_t\xto{a_1}t)\subset \Coh(G_t)$. 
  This is an immediate consequence of the base-change equivalence
  $$
    a_{1*}\circ a_2^* \simeq i_T^*\circ i_{T*}
  $$
  and of the regularity of $T$.
\end{rem}

\ssec{The main computation}

We use the above explicit models to prove \corref{useful loc seq}, which is the main ingredient for \thmref{devissage-like thm}.

\begin{rem}\label{Coh(G_t-->t)-->Coh(G_t)-->Sing(t) loc seq iff Sing(G_t-->t)-->Sing(G_t)-->Sing(t)}
The chain
  $$
    \Coh(G_t\xto{a_1} t) \hto \Coh(G_t) \to \Sing(t)
  $$
  is a localization sequence if and only if so is
  $$
    \Sing(G_t\xto{a_1} t) \hto \Sing(G_t) \to \Sing(t).
  $$
\end{rem}

\begin{lem}\label{Coh(G_t)/Coh(G_t -> t) is two periodic}
The dg-category 
  $$
    \frac{\Coh(G_t)}{\Coh(G_t\xto{a_1}t)}\simeq \frac{\Sing(G_t)}{\Sing(G_t\xto{a_1}t)} 
  $$
is $2$-periodic, i.e. there is a natural equivalence of dg-functors
  $$
    \id\simeq [2]: \frac{\Coh(G_t)}{\Coh(G_t\xto{a_1}t)}\to \frac{\Coh(G_t)}{\Coh(G_t\xto{a_1}t)}. 
  $$
Here $[2]=[1]\circ [1]$ denotes the double shift functor.
\end{lem}

\begin{proof}
We will show that, for every $E \in \Coh(G_t)$, there is a functorial exact triangle
$$
  a_2^* a_{2*}E\to E \xto{u_E} E[2]
$$
in $\Sing(G_t)$, where the first morphism is the counit of the adjunction $(a_2^*,a_{2*})$.\footnote{By abuse of notation we still denote $E\in \Sing(G_t)$ the image of $E\in \Coh(G_t)$ along $\Coh(G_t)\to \Sing(G_t)$.}

In the diagrams below, we will write horizontally from left to right the differentials of a complex and from right to left the homotopies which are part of the datum for an object in $\cohs\bigt{\scrO_{\uL},(\pi_{\uK},\pi_{\uK})}$. 
Morphisms in this category will be written vertically.
Notice that it suffices to consider an object in  
$\Sing(G_t)$ represented by some $(E,d,\{h_1,h_2\})\in \cohs\bigt{\scrO_{\uL},(\pi_{\uK},\pi_{\uK})}$ with $E$ concentrated in three degrees at most, say $[n,n+2]$ (see \cite[Theorem 2.7]{pi22b}):

\begin{equation*}
    \begin{tikzcd}
        E^n \arrow[r,shift right,swap,"d^n"]& \arrow[l,shift right,swap,"h_i^{n+1}"] E^{n+1} \arrow[r,shift right,swap,"d^{n+1}"] & \arrow[l,shift right,swap,"h_i^{n+21}"] E^{n+2}.
    \end{tikzcd}
\end{equation*}
In this case, using the strict models above, one computes that $a_2^*a_{2*}E\to E$ is
\begin{equation*}
    \begin{tikzcd}[column sep=huge, row sep=70, ampersand replacement = \&]
        E^n 
          \arrow[r,shift right,swap,"\begin{bmatrix} \pi_{\uK} \\ -d^n \end{bmatrix}"] 
      \& 
          \arrow[l,shift right,swap, "{\begin{bmatrix} 1 & 0 \end{bmatrix}}"]
          \arrow[l,bend right=60,shift right,swap, "{\begin{bmatrix} 0 & -h_2^{n+1} \end{bmatrix}}"]
        E^{n}\oplus E^{n+1} 
          \arrow[d,"{\begin{bmatrix} \id & h_1^{n+1} \end{bmatrix}}"]
          \arrow[r,shift right,swap,"{\begin{bmatrix} d^n & \pi_{\uK} \\ 0  & -d^{n+1}\end{bmatrix}}"] 
      \& 
          \arrow[l,shift right,swap,"{\begin{bmatrix}0 & 0\\ 1 & 0\end{bmatrix}}"]
          \arrow[l,bend right=60,shift right,swap, "{\begin{bmatrix}h_2^{n+1} & 0 \\ 0 & -h_2^{n+2} \end{bmatrix}}"]
        E^{n+1}\oplus E^{n+2} 
          \arrow[d,"{\begin{bmatrix} \id & h_1^{n+2}\end{bmatrix}}"] 
          \arrow[r,shift right,swap,"{\begin{bmatrix}d^{n+1} &  \pi_{\uK} \\ 0 & -d^{n+2} \end{bmatrix}}"] 
      \&
          \arrow[l,shift right,swap,"{\begin{bmatrix} 0 & 0\\ 1 & 0 \end{bmatrix}}"]  
          \arrow[l,bend right=60,shift right,swap, "{\begin{bmatrix}h_2^{n+2}  \\ 0  \end{bmatrix}}"]
        E^{n+2} 
          \arrow[d,"\id"] 
      \\
      \&
        E^n 
          \arrow[r,shift right,swap,"d^n"] 
      \&
          \arrow[l,shift right,swap,"h_i^{n+1}"] 
        E^{n+1} 
          \arrow[r,shift right,swap,"d^{n+1}"] 
      \& 
          \arrow[l,shift right,swap,"h_i^{n+21}"] 
        E^{n+2}.
    \end{tikzcd}
\end{equation*}
Then we have the following morphism from $E[2]$ to the cone of the morphism displayed above
\begin{equation*}
    \begin{tikzcd}[column sep=35, row sep=70, ampersand replacement = \&]
        E^n 
          \arrow[r,shift right,swap,"{\begin{bmatrix} -\pi_{\uK} \\ d^n \end{bmatrix}}"] 
      \& 
          \arrow[l,shift right,swap, "{\begin{bmatrix} -1 & 0 \end{bmatrix}}"]
          \arrow[l, bend right=60,shift right,swap, "{\begin{bmatrix} 0 & h_2^{n+1} \end{bmatrix}}"]
        E^{n}\oplus E^{n+1}        
           \arrow[r,shift right,swap,"{\begin{bmatrix}\id & h_1^{n+1}\\ -d^n & -\pi_{\uK} \\ 0  & d^{n+1}\end{bmatrix}}"] 
      \& 
          \arrow[l,shift right,swap,"{\begin{bmatrix}0 & 0 & 0\\ 0 & -1 & 0\end{bmatrix}}"]
          \arrow[l,bend right=65,shift right,swap, "{\begin{bmatrix}0 & -h_2^{n+1} & 0 \\ 0 & 0 & h_2^{n+2} \end{bmatrix}}"]
        E^{n}\oplus E^{n+1}\oplus E^{n+2}          
          \arrow[r,shift right,swap,"{\begin{bmatrix}d^{n} & \id & h_1^{n+2} \\ 0 &  -d^{n+1}  & -\pi_{\uK}\end{bmatrix}}"] 
      \&
          \arrow[l,shift right,swap,"{\begin{bmatrix} h_1^{n+1} & 0\\ 0 & 0 \\ 0 & -1 \end{bmatrix}}"]  
          \arrow[l,bend right=95,shift right,swap, "{\begin{bmatrix}h_2^{n+1} & 0  \\ 0 & -h_2^{n+2}\\ 0 & 0  \end{bmatrix}}"]
        E^{n+1}\oplus E^{n+2}    
          \arrow[r,shift right,swap,"{\begin{bmatrix} d^{n+1} & \id\end{bmatrix}}"]
      \&
          \arrow[l,shift right,swap,"{\begin{bmatrix} h_i^{n+2} \\ 0\end{bmatrix}}"]
          %\arrow[l,bend right=100,swap,"{\begin{bmatrix} h_2^{n+2} \\ 0\end{bmatrix}}"]
        E^{n+2}
       \\
        E^n 
          \arrow[r,shift right,swap,"d^n"] 
          \arrow[u,"-\id"] 
      \&
          \arrow[l,shift right,swap,"h_i^{n+1}"] 
        E^{n+1} 
          \arrow[r,shift right,swap,"d^{n+1}"] 
          \arrow[u,"{\begin{bmatrix} h_1^{n+1} \\-\id\end{bmatrix}}"] 
      \& 
         \arrow[l,shift right,swap,"h_i^{n+21}"] 
        E^{n+2}.
          \arrow[u,"{\begin{bmatrix} 0 \\ h_1^{n+2} \\ -\id\end{bmatrix}}"] 
      \&
      \&
    \end{tikzcd}
\end{equation*}
It is easy to check that the latter induces isomorphisms on cohomology groups. Now recall that the object $a_2^* a_{2*}E$ belongs to $\Coh(G_t\xto{a_1}t)$ and it is thus zero in the dg-quotient $\Coh(G_t)/\Coh(G_t\xto{a_1}t)$. 
In other words, the morphism
$$
  u_E:E\to E[2]
$$
becomes an equivalence in $\Coh(G_t)/\Coh(G_t\xto{a_1}t)$. 
This shows that there is a canonical equivalence $\id \simeq [2]$ on the dg-quotient of $\Coh(G_t\xto{a_1}t)\hto \Coh(G_t)$ computed in $\dgCats_S[\on{W}_{\on{qe}}^{-1}]$, the $\oo$-localization of $\dgCats_S$ with respect to quasi-equivalences.
The objects of this quotient dg-category are in correspondence with those of $\Coh(G_t)$.
Since $\Coh(G_t)/\Coh(G_t\xto{a_1}t)$ is Karoubi generated by the image of the canonical functor $\Coh(G_t)\to \Coh(G_t)/\Coh(G_t\xto{a_1}t)$ and the morphism $u_E$ is functorial in $E$, we obtain a natural equivalence
$$
  u:\id \xto{\sim} [2]: \frac{\Coh(G_t)}{\Coh(G_t\xto{a_1}t)}\to \frac{\Coh(G_t)}{\Coh(G_t\xto{a_1}t)}
$$
as claimed.
\end{proof}

\sssec{}

We will also need the following characterization of the objects in the quotient of ${\Coh(G_t)}$ by ${\Coh(G_t\xto{a_1}t)}$.

\begin{lem}\label{Coh(G_t)/Coh(G_t-->t) is generated by complexes concentrated in two degrees}
  The dg-category $\Coh(G_t)/\Coh(G_t\xto{a_1}t)$ is Karoubi generated by the images along 
  $$
    \cohs \bigt{\scrO_{\uL},(\pi_{\uK},\pi_{\uK})}\to \frac{\Coh(G_t)}{\Coh(G_t\xto{a_1}t)}
  $$
  of objects $(E,d,\{h_i\}_{i=1,2})$, where $(E,d)$ is a cochain complex concentrated in at most two degrees.
\end{lem}

\begin{proof}
  Thanks to the equivalence
  $$
    \frac{\Coh(G_t)}{\Coh(G_t\xto{a_1}t)}\simeq \frac{\Sing(G_t)}{\Sing(G_t\xto{a_1}t)},
  $$
  it follows from \cite[Theorem 2.7]{pi22b} that this dg-quotient is Karoubi generated by the images of objects $(E,d,\{h_i\}_{i=1,2})$, where $(E,d)$ is a cochain complex concentrated in at most three degrees:
  $$
    E^{n-1}\xto{d^{n-1}}E^n \xto{d^n} E^{n+1}.
  $$
  Now, $E^{n-1},E^n$ and $E^{n+1}$ are finitely generated projective $\scrO_{\uL}$-modules. 
  The characterization of finitely generated modules over a principal ideal domain guarantees that each is free of finite rank. 
  Now, $\Ker(d^n)\subset E^n$ and $\on{Im}(d^{n})\subset E^{n+1}$ are submodules of free $\scrO_{\uL}$-modules of finite rank. 
  Thus, they are both finitely generated and torsion free; 
  hence they are also free $\scrO_{\uL}$-modules of finite rank.
  It follows that the cochain complexes (with the obvious $\uK \bigt{\scrO_{\uL},(\pi_{\uK},\pi_{\uK})}$-module structures)
  $$
    E^{n-1}\xto{d^{n-1}} \Ker(d^n), \hspace{0.5cm} \on{Im}(d^{n})\hto E^{n+1}
  $$
  belong to $\cohs \bigt{\scrO_{\uL},(\pi_{\uK},\pi_{\uK})}$.
  
  Moreover, $(E,d,\{h_i\}_{i=1,2})$ is clearly an extension of $\on{Im}(d^{n})\hto E^{n+1}$ by $E^{n-1}\xto{d^{n-1}} \Ker(d^n)$.
  We conclude that $\Coh(G_t)/\Coh(G_t\xto{a_1}t)$ is Karoubi generated by the images of those objects $(E,d,\{h_i\}_{i=1,2}) \in \Coh\bigt{\scrO_{\uL},(\pi_{\uK},\pi_{\uK})}$  such that $(E,d)$ is a cochain complex concentrated in two degrees at most. 
\end{proof}

\begin{rem}\label{h_1=h_2 if the complex is concentrated in two degrees}
  Suppose that an object 
  $$
  (E,d,\{h_i\}_{i=1,2})=(E^n\xto{d}E^{n+1}) \in \cohs \bigt{\scrO_{\uL},(\pi_{\uK},\pi_{\uK})}
  $$
  is concentrated in two degrees. 
  Then $h_1=h_2$. 
  In fact, the equations imposed on $d$ and $h_i$, combined with the fact that $\pi_{\uK}\in \scrO_{\uL}$ is not a zero-divisor, imply that $d$ is injective and that
  $$
    d\circ h_1=d\circ h_2.
  $$
\end{rem}

\begin{lem}
  In \lemref{Coh(G_t)/Coh(G_t-->t) is generated by complexes concentrated in two degrees}, it suffices to consider those complexes generated in degrees $[0,1]$.
\end{lem}
\begin{proof}
  By the two-periodicity of $\Coh(G_t)/\Coh(G_t\xto{a_1}t)$ (see \lemref{Coh(G_t)/Coh(G_t -> t) is two periodic}) and by the proof of \lemref{Coh(G_t)/Coh(G_t-->t) is generated by complexes concentrated in two degrees}, we see that it suffices to show that for every object $A$ represented by an object $\bbE=(E,d,\{h_i\}_{i=1,2})$ of $\cohs\bigt{\scrO_{\uL},(\pi_{\uK},\pi_{\uK})}$ concentrated in two degrees, $A[1]$ can be represented by some object of $\cohs\bigt{\scrO_{\uL},(\pi_{\uK},\pi_{\uK})}$ \emph{concentrated in the same two degrees}.
  We consider the dg-functor
  $$
    \cohs \bigt{\scrO_{\uL},\pi_{\uK}}\to \cohs \bigt{\scrO_{\uL},(\pi_{\uK},\pi_{\uK})}
  $$
  $$
    (E,d,h) \mapsto (E,d,\{h,h\}).
  $$
  This is a strict model for the pushforward along the diagonal map $\delta:t\to G_t\simeq t\times_T t$. 
  As $a_1\circ \delta=\id_t$, this induces a dg-functor
  $$
    \Sing(t)\to \frac{\Coh(G_t)}{\Coh(G_t\xto{a_1}t)}.
  $$
  By \remref{h_1=h_2 if the complex is concentrated in two degrees}, every object $A$ as above is in the image of this functor. 
  Therefore, it suffices to compute $A[1]$ before applying the functor, i.e. in $\Sing(t)$. 
  Then \cite[Corollary 3.7]{pi22b} is exactly what we want.
\end{proof}

\begin{prop}\label{an useful localization sequence}
  The dg-functor $a_{1*}:\Coh(G_t)\to \Coh(t)$ induces an equivalence
  $$
    \frac{\Coh(G_t)}{\Coh(G_t\xto{a_1}t)}\xto{\sim}\Sing(t)
  $$
  in $\dgCat_S$
\end{prop}
\begin{proof}
  The dg-functor 
  $$
    \ffQ:=\frac{\Coh(G_t)}{\Coh(G_t\xto{a_1}t)}\to\Sing(t)
  $$
  is induced by the dg-functor of strict models
  $$
    a_{1*}:\cohs \bigt{\scrO_{\uL},(\pi_{\uK},\pi_{\uK})}\to \cohs (\scrO_{\uL},\pi_{\uK})
  $$
  $$
    (E,d,\{h_i\}_{i=1,2})\mapsto (E,d,h_1).
  $$
  Moreover, as the dg-categories involved are triangulated (see \cite{to11}), it suffices to show that it induces an equivalence at the level of homotopy categories:
  $$
    \ccF: \uh(\ffQ)\to \uh \bigt{\Sing(t)}.
  $$
  Let $\widehat{\ffQ}$ and $\widehat{\Sing(t)}$ denote the dg-quotients of $\Coh(G_t\xto{a_1})\hto \Coh(G_t)$ and $\Perf(t)\hto \Coh(t)$ computed in $\dgCats_S[\on{W}_{\on{qe}}^{-1}]$. 
  In particular, $\ffQ$ and $\Sing(t)$ are the triangulated hulls of $\widehat{\ffQ}$ and $\widehat{\Sing(t)}$ respectively.
  Recall that the homotopy category of $\widehat{\Sing(t)}$ has the following explicit description:
  $$
    \uh\bigt{\widehat{\Sing(t)}}\simeq \uh \bigt{\MF(\scrO_{\uL},\pi_{\uK})},
  $$
  where $\uh \bigt{\MF(\scrO_{\uL},\pi_{\uK})}$ denotes the triangulated category of \emph{matrix factorizations} (see \cite{orl04}). 
  This is the category whose objects are tuplets $\bbE=\bigt{E^0,E^1,d:E^0\to E^1,h:E^1\to E^0}$, where $E^0,E^1$ are projective $\scrO_{\uL}$-modules of finite rank and $d,h$ are $\scrO_{\uL}$-linear maps such that $d\circ h=\pi_{\uK} \cdot \id_{E_1}$ and $h\circ d=\pi_{\uK} \cdot \id_{E_0}$. 
  For two such objects $\bbE,\bbE'$, the $\scrO_{\uK}$-module of morphisms $\bbE\to \bbE'$ is the set of pairs of $\scrO_{\uL}$-linear maps $\phi=(\phi^0:E^0\to E^{'0},\phi^1:E^1\to E^{'1})$ commuting with $d,d',h,h'$ in the obvious sense, endowed with the obvious $\scrO_{\uK}$-module structure. 
  The shift functor is
  $$
    [1]:\uh \bigt{\MF(\scrO_{\uL},\pi_{\uK})} \to \uh \bigt{\MF(\scrO_{\uL},\pi_{\uK})}
  $$
  $$
    \bigt{E^0,E^1,d:E^0\to E^1,h:E^1\to E^0}\mapsto \bigt{E^1,E^0,-h:E^1\to E^0,-d:E^0\to E^1}
  $$
  and the cone of a morphism $\phi:\bbE \to \bbE'$ as above is
  $$
    \coFib(\phi)=\biggl (E^1\oplus E^{'0},E^0\oplus E^{'1}, \begin{bmatrix}
      \phi^1 & d' 
      \\
      -h & 0
    \end{bmatrix},\begin{bmatrix}
      0 & -d 
      \\
      h' & \phi^0
    \end{bmatrix}
    \biggr).
  $$
  The distinguished triangles are those isomorphic to triangles of the form
  $$
    \bbE \xto{\phi} \bbE'\to \coFib(\phi).
  $$
  The equivalence $\uh\bigt{\widehat{\Sing(t)}}\simeq \uh \bigt{\MF(\scrO_{\uL},\pi_{\uK})}$ is induced by the functor 
  $$
  (E,d,h)
  \mapsto 
  \Bigt{\bigoplus_{i\in \bbZ}E^{2i},\bigoplus_{i\in \bbZ}E^{2i+1},d+h,d+h}.
  $$
  See \cite[Corollary 3.11]{pi22b}.
  Therefore, we get a triangulated functor
  $$
    \ccG: \uh \bigt{\MF(\scrO_{\uL},\pi_{\uK})}\to \uh(\widehat{\ffQ})
  $$
  $$
    \bbE=(E^0,E^1,d,h)\mapsto \ccG(\bbE)=(E^0\xto{d} E^1,\{h,h\})
  $$
  $$
    (\phi^0,\phi^1)\mapsto (\phi^0,\phi^1),
  $$
  where the object $\ccG(\bbE)$ is concentrated in degrees $[0,1]$.
  The composition $\widehat{\ccF} \circ \ccG$ is then the identity functor. 
  Here $\widehat{\ccF}$ denotes the dg-functor
  $$
    \uh (\widehat{\ffQ})
    \to
    \uh \bigt{\MF(\scrO_{\uL},\pi_{\uK})}
  $$
  $$
    (E,d,\{h_i\}_{i=1,2})
    \mapsto
    \Bigt{\bigoplus_{i\in \bbZ}E^{2i},\bigoplus_{i\in \bbZ}E^{2i+1},d+h_1,d+h_1}.
  $$
  The composition $\ccG \circ \widehat{\ccF}$ is also equivalent to the identity, as it is so on those objects concentrated in degrees $[0,1]$, to which every object in $\uh (\widehat{\ffQ})$ is isomorphic.
  Thus, we have proved that
  $$
    \uh(\widehat{\ffQ}) \simeq \uh\bigt{\MF(\scrO_{\uL},\pi_{\uK})}
  $$
  and the claim of the proposition follows immediately.
\end{proof}

\begin{cor}
  The following are localization sequences in $\dgCat_S$:
  $$
    \Coh(G_t\xto{a_1}t)\hto \Coh(G_t) \to \Sing(t),
  $$
  $$
    \Sing(G_t\xto{a_1}t)\hto \Sing(G_t) \to \Sing(t).
  $$
\end{cor}

\begin{proof}
  By \remref{Coh(G_t-->t)-->Coh(G_t)-->Sing(t) loc seq iff Sing(G_t-->t)-->Sing(G_t)-->Sing(t)}, it suffices to show that the first one is a localization sequence. 
  This follows from \propref{an useful localization sequence} and from the observation that $\Coh(G_t\xto{a_1}t)$ is by definition the kernel of the dg-functor
  $$
    \Coh(G_t)\to \Sing(t)
  $$
induced by $a_{1*}$.
\end{proof}

\ssec{The structure of left \texorpdfstring{$\sB^+$}{B+}-modules} 

In this section, we show that the above localization sequences are compatible with the natural $\sB^+$-module structures.

\sssec{}

Recall from \cite{tv22} that $\Coh(G_t)$ and $\Coh(t)$ are both equipped with natural left actions of $\sB^+$. 
Since these actions preserve the full subcategories of perfect complexes, the quotient dg-categories $\Sing(G_t)$ and $\Sing(t)$ are left $\sB^+$-modules, too.

\sssec{}

Let $\cohs \bigt{\scrO_{\uK},(\pi_{\uK},\pi_{\uK})}$ be the strict model for $\Coh(G)$, defined (\emph{mutatis mutandis}) just as $\cohs \bigt{\scrO_{\uL},(\pi_{\uK},\pi_{\uK})}$. 
There is a pseudo-action
$$
  - \odot - : \cohs \bigt{\scrO_{\uK},(\pi_{\uK},\pi_{\uK})}\otimes \cohs \bigt{\scrO_{\uL},(\pi_{\uK},\pi_{\uK})} \to \cohs \bigt{\scrO_{\uL},(\pi_{\uK},\pi_{\uK})}
$$
$$
  (M,E)\mapsto M\odot E:= M\otimes_{\uK(\scrO_{\uK},\pi_{\uK})} E,
$$
where $M$ and $E$ are seen as a $\uK (\scrO_{\uK},\pi_{\uK})$-module by forgetting the actions of $h_1$ and $h_2$ respectively and by restricting scalars on $E$. 
Similarly, there is a pseudo-action
$$
  - \odot - : \cohs \bigt{\scrO_{\uK},(\pi_{\uK},\pi_{\uK})}\otimes \cohs (\scrO_{\uL},\pi_{\uK}) \to \cohs (\scrO_{\uL},\pi_{\uK})
$$
$$
  (M,E)\mapsto M\odot E:= M\otimes_{\uK (\scrO_{\uK},\pi_{\uK})} E
$$
where $M$ is seen as a $\uK (\scrO_{\uK},\pi_{\uK})$-module by forgetting the action of $h_1$ and by restring scalars on $E$.

These are strict models for the left $\sB^+$-module structures on $\Coh(G_t)$ and $\Coh(t)$, see \cite[Section 4.1]{tv22}.

\begin{lem}
  The dg-functor
  $$
    a_{1*}:\Coh(G_t)\to \Coh(t)
  $$
  is a morphism of left $\sB^+$-modules.
\end{lem}

\begin{proof}
  This follows immediately from the strict models. 
  In fact, as already mentioned above, $a_{1*}:\Coh(G_t)\to \Coh(t)$ can be \virg{strictified} by the dg-functor
  $$
    \cohs \bigt{\scrO_{\uL},(\pi_{\uK},\pi_{\uK})}\to \cohs (\scrO_{\uL},\pi_{\uK})
  $$
  which forgets the action of $h_2$. 
  This is obviously compatible with the pseudo-actions described above.
\end{proof}

\begin{cor}\label{useful loc seq}
  The sequence
  $$
    \Coh(G_t\xto{a_1}t)\hto \Coh(G_t)\to \Sing(t)
  $$
  is a localization sequence of left $\sB^+$-modules.
\end{cor}

\begin{proof}
  As $a_{1*}:\Coh(G_t)\to \Coh(t)$ and $\Coh(t)\to \Sing(t)$ are morphisms of left $\sB^+$-modules, their composition is $\sB^+$-linear too. 
  It remains to show that the inclusion $\Coh(G_t\xto{a_1}t)\hto \Coh(G_t)$ is $\sB^+$-linear. 
  Observe that as $a_{1*}:\Coh(G_t)\to \Coh(t)$ is $\sB^+$-linear and the action of $\sB^+$ on $\Coh(t)$ preserves $\Perf(t)\subset \Coh(t)$, the dg-category $\Coh(G_t\xto{a_1}t)$ inherits a left $\sB^+$-module structure from $\Coh(G_t)$. 
  It is then clear that $\Coh(G_t\xto{a_1}t)\subset \Coh(G_t)$ is $\sB^+$-linear.
\end{proof}

\section{A d\'evissage-like result}

In this section we compute the motivic realization of the dg-category $\Coh(X_t\xto{i_{X_T}} X_T)$.

\ssec{Another useful localization sequence} 

\sssec{}

Recall from \cite[Section 2.1]{tv22} and \cite[Section 4.4]{lu17} that there is a functor
$$
  \dgCat^{\sB^+}\times \dgCat_{\sB^+}\to \dgCat_S
$$
which sends a right $\sB^+$-module $\ccR$ and a left $\sB^+$-module $\ccL$ to
$$
  \ccR \otimes_{\sB^+}\ccL:=(\ccR \otimes_S \ccL)\otimes_{\sB^{+,\ue}}\sB^{+,\uL},
$$
where $\sB^{+,\ue}$ denotes the \virg{enveloping} algebra $\sB^{+,\rev}\otimes_S \sB^+$, and $\sB^{+,\uL}$ the dg-category $\sB^{+}$ endowed with its natural left $\sB^{+,\ue}$-module structure. 

\sssec{}

Also recall (see \cite[Remark 2.1.4, Proposition 4.1.7]{tv22}) that $\Coh(X_s)$ is cotensored over $\sB^+$, so that $\Coh(X_s)^{\op}$ is a right $\sB^+$-module. 
We can therefore consider the functor
$$
  \Coh(X_s)^{\op}\otimes_{\sB^+}-: \dgCat_{\sB^+}\to \dgCat_S.
$$

\begin{lem}\label{Coh(X_s)^op otimes_CB^+ - preserves loc seq}
The functor $\Coh(X_s)^{\op}\otimes_{\sB^+}-: \dgCat_{\sB^+}\to \dgCat_S$ sends localization sequences of left $\sB^+$-modules to localization sequences in $\dgCat_S$.
\end{lem}

\begin{proof}
  By \cite[Corollary 4.4.2.15]{lu17} (applied to $\ccC=\dgCat_S$, $A=C=\Perf(S)$ and $B=\sB^+$), we know that this functor preserves colimits. 
  In particular, if
  $$
    \ccL_1 \hto \ccL_2 \to \ccL_3
  $$
  is a localization sequence in $\dgCat_{\sB^+}$, then
  $$
\Coh(X_s)^{\op}\otimes_{\sB^+}\ccL_3\simeq \frac{\Coh(X_s)^{\op}\otimes_{\sB^+}\ccL_2}{\Coh(X_s)^{\op}\otimes_{\sB^+}\ccL_1}.
  $$
It remains to prove that the induced arrow $\Coh(X_s)^{\op}\otimes_{\sB^+}\ccL_1\to \Coh(X_s)^{\op}\otimes_{\sB^+}\ccL_2$ is fully-faithful. 
Objects in the tensor product $\Coh(X_s)^{\op}\otimes_{\sB^+}\ccL_1$ (resp. $\Coh(X_s)^{\op}\otimes_{\sB^+}\ccL_2$) are pairs $(E,L)$ where $E\in \Coh(X_s)$ and $L\in \ccL_1$ (resp. $L\in \ccL_2$). 
For two such objects $(E,L)$, $(E',L')$, the hom complex of morphisms from $(E,L)$ to $(E',L')$ in $\Coh(X_s)^{\op}\otimes_{\sB^+}\ccL_1$ (resp. $\Coh(X_s)^{\op}\otimes_{\sB^+}\ccL_2$) is computed as
  $$
    \Hom_{\Coh(X_s)^\op}(E,E')\otimes_{k[u]}\Hom_{\ccL_{1}}(L,L')
  $$
  $$
    \bigt{\text{resp. } \Hom_{\Coh(X_s)^\op}(E,E')\otimes_{k[u]}\Hom_{\ccL_{2}}(L,L')}.
  $$
  Here $k[u]$ is the algebra of endomorphisms of the unit object of $\sB^+$. 
  
  For two objects $L,L' \in \ccL_1$, the morphism
  $$
    \Hom_{\ccL_{1}}(L,L')\to \Hom_{\ccL_{2}}(L,L')
  $$
  is a quasi-isomorphism. 
  Therefore, its image along
  $$
    \Hom_{\Coh(X_s)^\op}(E,E')\otimes_{k[u]}-
  $$
  is a quasi-isomorphism as well and the claim follows.
\end{proof}

\begin{cor}\label{useful loc seq 2 preliminary}
  The sequence
  $$
    \Coh(X_s)^{\op}\otimes_{\sB^+}\Coh(G_t\xto{a_1}t)\hto \Coh(X_s)^{\op}\otimes_{\sB^+}\Coh(G_t)\to \Coh(X_s)^{\op}\otimes_{\sB^+}\Sing(t)
  $$
  is a localization sequence in $\dgCat_S$.
\end{cor}

\begin{proof}
  This follows immediately from \lemref{Coh(X_s)^op otimes_CB^+ - preserves loc seq} and \corref{useful loc seq}.
\end{proof}

\sssec{}

Our next goal is to identify the localization sequence above with a more explicit one. 
Recall from \cite[Section 4.2]{tv22} that, for two regular and flat $S$-schemes $Y$ and $Z$, there is an equivalence
$$
  \ol \ffF_{Y,Z}:= j_*\bigt{\bbD_{Y_s}(-)\boxtimes_s(-)}: \Coh(Y_s)^\op \otimes_{\sB^+}\Coh(Z_s)\xto{\sim} \Coh(Y\times_S Z)_{Y_s\times_s Z_s},
$$
where $-\boxtimes_s-$ denotes the external tensor product over $s$, $\bbD_{Y_s}(-):=\ul \Hom_{Y_s}(-,\ccO_{Y_s})$ the Grothendieck duality functor and $\Coh(Y\times_S Z)_{Y_s\times_s Z_s}$ the subcategory of $\Coh(Y\times_S Z)$ spanned by those complexes supported on the closed subscheme $j:Y_s\times_s Z_s\hto Y\times_S Z$. 

\begin{rem}\label{hypotesis needed on ol fF equivalence}
  Actually, we observe that the proof of \cite[Lemma 4.2.3]{tv22} only requires that $Y$ and $Z$ are Gorenstein $S$-schemes of finite type. 
  The flatness assumption is never really used (it is actually there only to guarantee that the derived special fibers agree with the usual ones), while regularity is only needed to guarantee that this functor is compatible with perfect complexes and thus induces a functor at the level of dg-categories of singularities.
\end{rem}

\begin{prop}\label{useful loc seq 2}
The localization sequence
$$
  \Coh(X_s)^\op \otimes_{\sB^+} \Bigt{\Coh(G_t\xto{a_1}t)\hto \Coh(G_t)\to \Sing(t)}
$$
of \corref{useful loc seq 2 preliminary} identifies with
$$
  \Coh(X_t\xto{i_{X_T}} X_T)\hto \Coh(X_t)\xto{i_{X_T*}} \Sing(X_T).
$$
\end{prop}

\begin{proof}
We will proceed in steps.

\sssec*{Step 1}

By \remref{hypotesis needed on ol fF equivalence}, we have an equivalence
  $$
    \ol \ffF_{X,t}:\Coh(X_s)^\op \otimes_{\sB^+} \Coh(G_t)\xto{\sim} \Coh(X\times_St)_{X_s\times_s G_t}.
  $$
However, $X_s\times_s G_t\hto X\times_S t\simeq X_t$ is a closed embedding with empty open complement (as $G_t\xto{a_1}t$ is so). 
In particular,
  $$
    \Coh(X\times_St)_{X_s\times_s G_t}\simeq \Coh(X_t)
  $$
and thus we obtain an equivalence
$$
\ol \ffF_{X,t}:\Coh(X_s)^\op \otimes_{\sB^+} \Coh(G_t)\xto{\sim} \Coh(X_t).
$$

\sssec*{Step 2}
Starting from the localization sequence 
  $$
    \Perf(t)\hto \Coh(t)\to \Sing(t)
  $$
of left $\sB^+$-modules, \lemref{Coh(X_s)^op otimes_CB^+ - preserves loc seq} implies that
  $$
    \Coh(X_s)^\op\otimes_{\sB^+}\Perf(t)\to \Coh(X_s)^\op\otimes_{\sB^+}\Coh(t)\to \Coh(X_s)^\op\otimes_{\sB^+}\Sing(t)
  $$
is a localization sequence in $\dgCat_S$. 
As explained in the proofs of \cite[Lemma 4.2.3, Theorem 4.2.1]{tv22}, the functor $\ol \ffF_{X,T}$ yields a commutative diagram
\begin{equation*}
  \begin{tikzcd}[column sep=huge, row sep=large]
    \Coh(X_s)^\op \otimes_{\sB^+}\Perf(t)
       \arrow[r]
       \arrow[d,"\ol \ffF_{X,t}"]
    &
      \Coh(X_s)^\op \otimes_{\sB^+}\Coh(t)
        \arrow[d,"\ffF_{X,T}"]
    \\
      \Perf(X_T)_{X_t}
        \arrow[r]
    &
      \Coh(X_T)_{X_t},
  \end{tikzcd}
\end{equation*}
where the vertical arrows are equivalences. 
Therefore,
  $$
    \Coh(X_s)^\op \otimes_{\sB^+}\Sing(t)\simeq \frac{\Coh(X_T)_{X_t}}{\Perf(X_T)_{X_t}}\simeq \Sing(X_T).
  $$
Here we used that $X_T$ has smooth generic fiber, so that
  $$
    \frac{\Coh(X_T)_{X_t}}{\Perf(X_T)_{X_t}}\simeq \frac{\Coh(X_T)}{\Perf(X_T)}=\Sing(X_T).
  $$
Thus, $\ol\ffF_{X,T}$ induces an equivalence
$$
\ffF_{X,T}:\Coh(X_s)^\op \otimes_{\sB^+}\Sing(t)
\xto{ \simeq } 
\Sing(X_T).
$$

\sssec*{Step 3}

Consider now the square
\begin{equation*}
  \begin{tikzcd}[column sep=huge, row sep=large]
     \Coh(X_s)^\op \otimes_{\sB^+}\Coh(G_t)
       \arrow[r,"\id\otimes_{\sB^+}a_{1*}"]
       \arrow[d,"\ol \ffF_{X,t}"]
    &
      \Coh(X_s)^\op \otimes_{\sB^+}\Sing(t)
        \arrow[d,"\ffF_{X,T}"]
    \\
      \Coh(X_t)
        \arrow[r,"i_{X_T*}"]
    &
      \Sing(X_T),
  \end{tikzcd}
\end{equation*}
and recall from the two previous steps that both vertical arrows are equivalences. In this step, we prove that this square is commutative.
So, we need to show that dg-functor $\Coh(X_t)\to \Sing(X_T)$ corresponding to 
$$
  \Coh(X_s)^\op \otimes_{\sB^+}\bigt{\Coh(G_t)\to \Sing(t)}
$$ 
identifies with the pushforward along $i_{X_T}:X_t\hto X_T$ composed with the projection 
$$
  \Coh(X_T) \tto \Sing(X_T).
$$
For this, we will use the diagram
\begin{equation*}
  \begin{tikzcd}[column sep=huge, row sep=large]
      X\times_Ss
      &
      X\times_Ss
      \\
        \arrow[u,"\pr_{12}"]
      X\times_Ss\times_Ss\times_S T
        \arrow[r,"\pr_{134}"]
        \arrow[d,"\pr_{234}"]
      &
        \arrow[u,"\tilde{\pr}_{12}"]
      X\times_Ss\times_S T
        \arrow[r,"i_{X_T}=\tilde{\pr}_{13}"]
        \arrow[d,"\tilde{\pr}_{23}"]
      &
      X\times_S T
        \arrow[d,"\pr_2"]
      \\
      s\times_Ss\times_S T
        \arrow[r,"a_1=\hat{\pr}_{13}"]
      &
      s\times_S T
        \arrow[r,"i_T"]
      &
      T
  \end{tikzcd}
\end{equation*}
Recall that $\ol \ffF_{X,T}$ is defined as
$$
(\tilde{p}_{13})_*\bigt{\tilde{\pr}_{12}^*\bbD_{X_s}(-)\otimes \tilde{\pr}_{23}^*(-)}
  :
\Coh(X\times_Ss)^{\op}\otimes_{\sB^+}\Coh(s\times_ST)
  \to
  \Coh(X\times_ST)_{X\times_Ss\times_ST}
$$
and that $\ol \ffF_{X,s\times_ST}$ is defined as
$$
  (\pr_{134})_*\bigt{\pr_{12}^*\bbD_{X_s}(-)\otimes \pr_{234}^*(-)}
  :
\Coh(X\times_Ss)^{\op}\otimes_{\sB^+}\Coh(s\times_Ss\times_ST)
  \to
  \Coh(X\times_Ss\times_ST).
$$
Let $E\in \Coh(X_s)$ and $M\in \Coh(G_t)$. To identify $\ol \ffF_{X,T}(E,a_{1*}M)$ with $i_{X_T*} \ol \ffF_{X,t}(E,M)$, we compute:
\begin{eqnarray}
\nonumber
  i_{X_T*} \ol \ffF_{X,s\times_ST}(E,M) & = & (\tilde{\pr}_{13})_*(\pr_{134})_*\bigt{\pr_{12}^*\bbD_{X_s}(E)\otimes \pr_{234}^*M} 
  \\
  \nonumber
  & \simeq & (\tilde{\pr}_{13})_*(\pr_{124})_*\bigt{\pr_{124}^*\tilde{\pr}_{12}^*\bbD_{X_s}(E)\otimes \pr_{234}^*M}  \\
  \nonumber
  & \simeq & (\tilde{\pr}_{13})_*\bigt{\tilde{\pr}_{12}^*\bbD_{X_s}(E)\otimes (\pr_{124})_*\pr_{234}^*M}
  \\
  \nonumber
  & \simeq & (\tilde{\pr}_{13})_*\bigt{\tilde{\pr}_{12}^*\bbD_{X_s}(E)\otimes \tilde{\pr}_{23}^*(\hat{\pr}_{13})_*M}
  \\
  \nonumber
  & = & \ffF_{X,T}(E,a_{1*}M) \hspace{0.5cm}\text{(when considered as an object in $\Sing(X_T)$)}.
\end{eqnarray}
The first equivalence follows from the obvious identities
$$
\tilde{\pr}_{13}\circ \pr_{124}\simeq \tilde{\pr}_{13}\circ \pr_{134}
$$
$$
\pr_{12}\simeq \tilde{\pr}_{12}\circ \pr_{124},
$$
the second equivalence is the projection formula and the last equivalence from the base-change
$$
\tilde{\pr}_{23}^*(\hat{\pr}_{13})_*\simeq (\pr_{124})_*\pr_{234}^*.
$$
This shows that the diagram of dg-functors
\begin{equation*}
  \begin{tikzcd}[column sep=huge, row sep=large]
    \Coh(X\times_Ss)^\op \otimes_{\sB^+}\Coh(s\times_Ss\times_ST)
      \arrow[r,"\id \otimes a_{1*}"]
      \arrow[d,"\ol \ffF_{X,t}"]
    &
    \Coh(X\times_Ss)^\op \otimes_{\sB^+}\Coh(s\times_ST)
      \arrow[d,"\ol \ffF_{X,T}"]
    \\
    \Coh(X\times_Ss\times_ST)
      \arrow[r,"i_{X_T*}"]
    &
    \Coh(X\times_ST)_{X\times_Ss\times_ST}
  \end{tikzcd}
\end{equation*}
is commutative. As an immediate consequence, we obtain
$$
  \Coh(X_s)^\op \otimes_{\sB^+}\bigt{\Coh(G_t)\to \Sing(t)}
  \simeq
  \bigt{\Coh(X_t)\xto{i_{X_T*}}\Coh(X_T)\tto \Sing(X_T)}.
$$

\sssec*{Step 4} %

We can now conclude the proof. The commutativity of the above square implies that
  $$
    \Coh(X_s)^\op \otimes_{\sB^+} \Coh(G_t\xto{a_1}t)\simeq \Coh(X_t\xto{i_{X_T}} X_T)
  $$
as the left-hand side has to be equivalent to
  $$
    \Ker \bigt{\Coh(X_t)\xto{i_{X_T*}}\Sing(X_T)}=\Coh(X_t\xto{i_{X_T}}X_T
    .
  $$
\end{proof}

\ssec{Motivic realization of \texorpdfstring{$\Coh(X_t\xto{i} X_T)$}{Coh(Xt to XT)}} 

\sssec{}

The motivic spectrum underlying $\Mv_S\bigt{\Coh(X_t\xto{i_{X_T}}X)}$ is easy to describe: it is the functor
$$
  \Sm_S^\op \to \Sp
$$
that sends a smooth $S$-scheme $Y$ to $\HK\bigt{\Coh(X_t\xto{i_{X_T}}X)\otimes_{\Perf(S)}\Perf(Y)}$.

\sssec{}
 
We think of the following statement as a kind of \emph{d\'evissage for homotopy-invariant non-connective algebraic $\uK$-theory}. 
We therefore state it as a theorem, as it seems to be a new result interesting on its own.
\begin{thm}\label{devissage-like thm}
With the same notation as in the previous sections, there are equivalences
$$
  \Mv_S\bigt{\Coh(X_t\xto{i_{X_T}}X_T)}\simeq \Mv_S\bigt{\Perf(X_T)_{X_t}}\simeq q_{T*}i_{X_T*}i_{X_T}^!\BU_{X_T}
$$
in $\SH_S$, where $q_T:X_T\to S$ is the composition $X_T\xto{p_T} T \to S$.
\end{thm}

\begin{proof}
\propref{useful loc seq 2} immediately yields a commutative diagram 
\begin{equation*}
  \begin{tikzcd}[column sep=huge, row sep=large]
    \Coh(X_t\xto{i_{X_T}}X_T)
    \arrow[r]
    \arrow[d,"i_{X_T*}"]
    &
    \Coh(X_t)
    \arrow[r]
    \arrow[d,"i_{X_T*}"]
    &
    \Sing(X_T)
    \arrow[d,"\id"]
    \\
    \Perf(X_T)_{X_t}
    \arrow[r]
    &
    \Coh(X_T)_{X_t}
    \arrow[r]
    &
    \Sing(X_T),
  \end{tikzcd}
\end{equation*}
where the rows are localization sequences in $\dgCat_S$. 
As $\Mv_S$ sends localization sequences to fiber-cofiber sequences, we obtain a commutative diagram
\begin{equation}\label{diag: proof computation K theory relative perfect complexes}
  \begin{tikzcd}[column sep=huge, row sep=large]
    \Mv_S\bigt{\Coh(X_t\xto{i_{X_T}}X_T)}
    \arrow[r]
    \arrow[d,"\Mv_S\bigt{i_{X_T*}}"]
    &
    \Mv_S\bigt{\Coh(X_t)}
    \arrow[r]
    \arrow[d,"\Mv_S\bigt{i_{X_T*}}"]
    &
    \Mv_S\bigt{\Sing(X_T)}
    \arrow[d,"\id"]
    \\
    \Mv_S\bigt{\Perf(X_T)_{X_t}}
    \arrow[r]
    &
    \Mv_S\bigt{\Coh(X_T)_{X_t}}
    \arrow[r]
    &
    \Mv_S\bigt{\Sing(X_T)},
  \end{tikzcd}
\end{equation}
where the rows are fiber-cofiber sequences. We know that 
  $$
  \Mv_S(i_{X_T*}):\Mv_S\bigt{\Coh(X_t)}\to \Mv_S\bigt{\Coh(X_T)_{X_t}}
  $$
is an equivalence. 
For $Y$ a smooth $S$-scheme, the map of spectra 
    $$
      \Mv_S\bigt{\Coh(X_t)}(Y)\xto{\Mv_S(i_{X_T*})(Y)} \Mv_S\bigt{\Coh(X_T)_{X_t}}(Y)
    $$
identifies with a map
    $$
      \HK \bigt{\Coh(X_t)\otimes_{\Perf(S)}\Perf(Y)}\to \HK \bigt{\Coh(X_T)_{X_t}\otimes_{\Perf(S)} \Perf(Y)}.
    $$
Since $Y$ is a smooth $S$-scheme, $\Perf(Y)\simeq \Coh(Y)$ and it follows from \cite[Proposition B.4.1]{pr11} that
    $$
      \Coh(X_t)\otimes_{\Perf(S)} \Perf(Y)\simeq \Coh(X_t\times_SY)\simeq \Coh(X_t\times_s Y_s).
    $$
Similarly, using the fact that $-\otimes_{\Perf(S)}\Perf(Y)$ preserves localization sequences (see \cite[Proposition 3.19 2)]{ro15}), \cite[Proposition B.4.1]{pr11} implies that $\Coh(X_T)_{X_t}\otimes_{\Perf(S)} \Perf(Y)$ identifies with the kernel of the localization dg-functor
    $$
      \Coh(X_T\times_S Y)\to \Coh(X_{\uL}\times_S Y)\simeq \Coh(X_{\uL}\times_{\eta}Y_{\eta}),
    $$
that is, with $\Coh(X_T\times_S Y)_{X_t\times_sY_s}$.
It follows that the map $\Mv_S(i_{X_T*})(Y)$ identifies with 
    $$
      \underbrace{\HK\bigt{\Coh(X_t\times_s Y_s)}}_{=\uG(X_t\times_s Y_s)}\xto{\HK\bigt{(X_t\times_s Y_s\to X_T\times_S Y)_*}}\underbrace{\HK\bigt{\Coh(X_T\times_S Y)_{X_t\times_sY_s}}}_{=\uG(X_T\times_S Y)_{X_t\times_s Y_s}},
    $$
which is an equivalence by the theorem of the heart (see \cite{bar15,nee01,nee98,nee99}) and by d\'evissage in $\uG$-theory (see \cite[\S5, Theorem 4]{qu73}).
    
Since the middle and rightmost vertical arrows in diagram \eqref{diag: proof computation K theory relative perfect complexes} are equivalences,
  $$
\Mv_S(i_{X_T*}):\Mv_S\bigt{\Coh(X_t\xto{i_{X_T}}X_T)}\to \Mv_S\bigt{\Perf(X_T)_{X_t}}
  $$ 
is an equivalence as well. 
To show that these motivic spectra identify with $q_{T*}i_{X_T*}i_{X_T}^!\BU_{X_T}$, we consider the localization sequence
  $$
    \Perf(X_T)_{X_t}\hto \Perf(X_T) \to \Perf(X_{\uL}).
  $$
Combining this with the equivalence
  $$
    \Mv_S\bigt{\Perf(X_T) \to \Perf(X_{\uL})}\simeq q_{T*}\bigt{\BU_{X_T}\to j_{X_T*}\BU_{X_{\uL}}}
  $$
(see \cite{brtv18}), we deduce that 
  $$
    \Mv_S\bigt{\Perf(X_T)_{X_t}\hto \Perf(X_T) \to \Perf(X_{\uL})}
  $$
is a fiber-cofiber sequence that identifies with the localization sequence
  $$
    q_{T*}i_{X_T*}i_{X_T}^!\BU_{X_T}\to q_{T*}\BU_{X_T}\to q_{T*}j_{X_T*}\BU_{X_{\uL}}
  $$
associated to the open-closed decomposition $i_{X_T}:X_t\to X_T \lto X_{\uL}:j_{X_T}$.
\end{proof}

\begin{rem}
The above theorem agrees with the prediction, stated in \cite{brtv18,pi22}, that 
  $$
    \Mv_S\bigt{\Coh(Y_s\xto{i_Y} Y)}\simeq \Mv_S\bigt{\Perf(Y)_{Y_s}}
  $$
  for every qcqs flat $S$-scheme $Y$ of finite type.
\end{rem}

\section{The main theorems}

As already mentioned in Section \ref{introduction} (and as already pointed out in \cite[Remark 4.46]{brtv18}), once \thmref{devissage-like thm} is established, the proof of \cite[Theorem 4.39]{brtv18} works essentially unchanged. In this section, we spell out the minor changes needed for the proof of \thmref{main thm A}.

\ssec{\texorpdfstring{$\ell$}{l}-adic vanishing cycles and \texorpdfstring{$\IL$}{IL}-homotopy fixed points} 
We remind here the definition of the vanishing cohomology introduced in \cite{sga7ii,sga7i}.

\sssec{}

Let $\bar{j}_X: X_{\ol \uK}\to X$ be the pullback of $\bareta \to S$ along $p:X\to S$.
Denote by $\shv_{\Ql{}}(X_s)^{\IK}$ the $\oo$-category of $\ell$-adic sheaves on $X_s$ endowed with a continuous action of $\IK$.
The functor of \emph{nearby cycles} is defined by
$$
  \Psi_p: \shv_{\Ql{}}(X_{\uK})\to \shv_{\Ql{}}(X_s)^{\IK}
$$
$$
 E \mapsto i_X^*\bar{j}_{X*}(E_{|X_{\ol \uK}}),
$$
with the $\IK$-action induced by transport of structure from the natural $\IK$-action on $E_{|X_{\ol \uK}}$.

\begin{rem}
We do not spell out the details of this construction here.
These are provided for example in \cite{cp22} for finite coefficients. Then one can take a limit and invert $\ell$ to get $\Ql{}$-coefficients.
\end{rem}

\sssec{}

For an $\ell$-adic sheaf $E$ on $X$, there is a functorial morphism
$$
  \on{sp}_E: i_X^*(E)\to \Psi_p(E_{|X_{\uK}})
$$
called the \emph{specialization morphism}, induced by the counit of the adjunction $(\bar{j}_X^*,\bar{j}_{X*})$. 
This morphism is $\IK$-equivariant if we endow $i_X^*(E)$ with the trivial $\IK$-action.

\sssec{} 

The \emph{vanishing cycles} functor
$$
  \Phi_p: \shv_{\Ql{}}(X)\to \shv_{\Ql{}}(X_s)^{\IK}
$$
is defined as
$$
  \Phi_p(E):=\coFib(\on{sp}_E),
$$
where the cofiber is computed in $\shv_{\Ql{}}(X_s)^{\IK}$.

\sssec{}

Let us recall an explicit description of the homotopy $\IL$-fixed points of $\Phi_p(\Ql{,X})$. 
Let $v_X:X_T\to X$ be the pullback of $T\to S$ along $p:X\to S$.

\begin{lem}\label{I_L-invariant vanishing cycles}
There is a canonical equivalence
$$
  \Phi_p(\Ql{,X})^{\IL}\simeq \coFib \bigt{\Ql{,X_s}\oplus \Ql{,X_s}(-1)[-1]\to i_X^*v_{X*}j_{X_T*}\Ql{,X_{\uL}}},
$$
compatible with the natural actions of $\HL$ on both sides.
\end{lem}

\begin{proof}
As taking $\IL$-fixed points is an exact functor, we have an equivalence
$$
  \Phi_p(\Ql{,X})^{\IL}\simeq \coFib \bigt{\Ql{,X_s}^{\IL}\xto{\on{sp}_{\Ql{,X}}^{\IL}}\Psi_p(\Ql{,X_{\uK}})^{\IL}}.
$$
We will start by proving that there are equivalences
$$
  \Ql{,X_s}\oplus \Ql{,X_s}(-1)[-1]\simeq \Ql{,X_s}^{\IL},
$$
$$
  \Psi_p(\Ql{,X_{\uK}})^{\IL}\simeq i_X^*v_{X*}j_{X_T*}\Ql{,X_{\uL}},
$$
compatible with the $\HL$-actions.
In the first equivalence, both members are equipped with the trivial $\HL$-action; in the second one, the first member carries the canonical $\HL$-action and the second one the action induced by transport of structure.

The first equivalence follows from the computations carried out in \cite{brtv18} (applied to $p_T:X_T\to T$). 
This is tautologically compatible with the $\HL$-actions, as these actions are trivial on both sides.

The second equivalence is a form of Galois descent as in \cite[Proposition 4.31]{brtv18}.
Indeed, the morphism
$u_{\uL}:X_{\bar \uK}\to X_{\uL}$ induces an equivalence
$$
  u_{\uL}^*:\shv_{\Ql{}}(X_{\uL})^{\HL}\leftrightarrows \shv_{\Ql{}}(X_{\bar \uK})^{\IK}: u_{\uL*}(-)^{\IL}
$$
between the $\oo$-category of $\Ql{}$-adic sheaves on $X_{\uL}$ endowed with a $\HL$-action and the $\oo$-category of $\Ql{}$-adic sheaves on $X_{\bar \uK}$ endowed with a continuous $\IK$-action.
In particular, $\Ql{,X_{\uL}}\simeq u_{\uL*}(\Ql{,X_{\bar \uK}})^{\IL}$.
It follows that
$$
i_X^*v_{X*}j_{X_T*}\Ql{,X_{\uL}}\simeq i_X^*v_{X*}j_{X_T*}u_{\uL*}(\Ql{,X_{\bar \uK}})^{\IL}\simeq i_X^*(v_{X*}j_{X_T*}u_{\uL*}\Ql{,X_{\bar \uK}})^{\IL},
$$
where the latter equivalence holds since the functor $(-)^{\IL}$ commutes with pushforwards. 
Using the continuity of the $\IL$-action as in \cite[Proposition 4.31]{brtv18}, we deduce that
$$
(i_X^*v_{X*}j_{X_T*}u_{\uL*}\Ql{,X_{\bar \uK}})^{\IL}=\Psi_p(\Ql{,X_{\uL}})^{\IL}.
$$
It remains to construct an homotopy between $ \on{sp}_{\Ql{,X}}^{\IL}$ and
$$
\Ql{,X_s}\oplus \Ql{,X_s}(-1)[-1]\to i_X^*v_{X*}j_{X_T*}\Ql{,X_{\uL}}.
$$
For this, it suffices to observe that both morphisms pre-composed with $\Ql{,X_s}\to \Ql{,X_s}\oplus \Ql{,X_s}(-1)[-1]$ are homotopic to the morphism $\Ql{,X_s}\to i_X^*v_{X*}j_{X_T*}\Ql{,X_{\uL}}$ induced by the unit of the adjunction $((v_X\circ j_{X_T})^*,(v_X\circ j_{X_T})_*)$. 
Now, the fact that both $\on{sp}_{\Ql{,X}}^{\IL}$ and $\Ql{,X_s}\oplus \Ql{,X_s}(-1)[-1]\to i_X^*v_{X*}j_{X_T*}\Ql{,X_{\uL}}$ are $\Ql{,X_s}^{\IL}$-linear concludes the proof. 
Notice that the latter morphism is $\HL$-equivariant, as it factors through $i_X^*j_{X*}\Ql{X_{\uK}}\simeq \Psi_p(\Ql{,X_{\uK}})^{\IK}\simeq \bigt{\Psi_p(\Ql{,X_{\uK}})^{\IL}}^{\HL}$.
\end{proof}

\begin{rem}
    A similar result holds (with the same proof) if we replace $\Ql{,X}$ with $\Ql{,X}(\beta)$.
\end{rem}

\ssec{The action of \texorpdfstring{$\HL$}{HL}}

\sssec{}

Recall that $\HL$ denotes the (finite) group $\IK/\IL$. 
Explicitly, $\scrO_{\uL}$ is isomorphic to the quotient of the polynomial ring $\scrO_{\uK}[x]$ by an Eisenstein polynomial $E(x)\in \scrO_{\uK}[x]$ of degree $\ue$. 
The group $\HL \simeq \Gal(\uL/\uK)$ permutes the roots of $E(x)$ and thus acts on $\scrO_{\uL}$.

We thus obtain actions of $\HL$ on the $S$-schemes $T$, $X_T$, $\eta_{\uL}$, $X_{\uL}$, $t$ and $X_t$. 
These actions are compatible in the natural way.

\sssec{}

We obtain actions (induced by pullbacks) of $\HL$ on $\Coh(X_T)$, $\Coh(X_t)$, $\Coh(X_t\xto{i_{X_T}}X_T)$, etc.
In turn, these immediately yield actions of $\HL$ on the motivic and $\ell$-adic realizations of such dg-categories.

\begin{lem}

There is an $\HL$-equivariant equivalence 
  $$
    \Mv_S\bigt{\Perf(X_{\uL})}\simeq q_{T*}j_{X_T*}\BU_{X_{\uL}}
  $$
in $\Mod_{\BU_S}(\SH_S)$. In particular, there is an $\HL$-equivariant equivalence
  $$
  \rl_S\bigt{\Perf(X_{\uL})}\simeq q_{T*}j_{X_T*}\Ql{,X_{\uL}}(\beta)
  $$
in $\Mod_{\Ql{,S}(\beta)}\bigt{\shv_{\Ql{}}(S)}$.

\end{lem}

\begin{proof}

The first equivalence is one of the main features of the motivic realization of dg-categories. 
By functoriality, it is obviously compatible with the actions of $\HL$: these are both induced by the $\HL$-action on $X_{\uL}$. 
The second equivalence follows immediately from the first one.
\end{proof}

\begin{lem}\label{l-adic realization Perf(X_t)}

There is an equivalence 
  $$
    \rl_S\bigt{\Perf(X_t)}\simeq p_{*}i_{X*}\Ql{,X_s}(\beta)
  $$
in $\Mod_{\Ql{,S}(\beta)}\bigt{\shv_{\Ql{}}(S)}$. 
The group $\HL$ acts trivially on both sides.

\end{lem}

\begin{proof}
Notice that $r:X_s\to X_t$ is a closed embedding (induced by $s=(t)_{\on{red}}\to t$) with empty open complement. 
The localization sequence in $\ell$-adic cohomology implies that $r_*\Ql{,X_s}(\beta)\simeq \Ql{,X_t}(\beta)$. 
Moreover, we have that
  $$
    \rl_S\bigt{\Perf(X_t)}\simeq q_{T*}i_{X_T*}\Ql{,X_t}(\beta), \hspace{0.5cm} \rl_S\bigt{\Perf(X_s)}\simeq p_{*}i_{X*}\Ql{,X_s}(\beta).
  $$
Then the desired equivalence follows from $q_T\circ i_{X_T}\circ r = p\circ i_X$. 

It remains to show that $\HL$ acts trivially on $\rl_S\bigt{\Perf(X_t)}$. 
This is clear: the action is induced by pullbacks along the isomorphisms $h:X_t\to X_t$, which verify the equations $r=h\circ r$.
\end{proof}

\ssec{The \texorpdfstring{$\ell$}{l}-adic realization of \texorpdfstring{$\Sing(X_t\xto{i_{X_T}}X_T)$}{Sing(Xt to XT)}}

We now approach the proof of our main theorem. 
\begin{prop}\label{l-adic realization of Sing(X_t-->X_T)}
The $\ell$-adic realization of $\Sing(X_t\xto{i_{X_T}}X_T)$ lives naturally in the following fiber-cofiber sequence:
  $$
    \rl_S\bigt{\Sing(X_t\xto{i_{X_T}}X_T)}\to p_*i_{X_*}\bigt{\Ql{,X_s}(\beta)\oplus \Ql{,X_s}(\beta)[1]} \to q_{T*}i_{X_T*}i_{X_T}^*j_{X_T*}\Ql{,X_{\uL}}(\beta).
  $$
Here, $\HL$ acts trivially on the middle term and naturally on the right one.
\end{prop}

\begin{proof}
By applying $\rl_S$ to the localization sequence
  $$
    \Perf(X_t)\to \Coh(X_t\xto{i_{X_T}}X_T)\to \Sing(X_t\xto{i_{X_T}}X_T),
  $$
together with \thmref{devissage-like thm} and \lemref{l-adic realization Perf(X_t)}, we get the fiber-cofiber sequence 
  $$
    p_*i_{X*}\Ql{,X_s}(\beta)\to q_{T*}i_{X_T*}i_{X_T}^!\Ql{,X_T}(\beta)\to \rl_S\bigt{\Sing(X_t\xto{i_{X_T}}X_T)}
  $$
in $\Mod_{\Ql{,S}(\beta)}\bigt{\shv_{\Ql{}}(S)}$.
In particular, we observe that 
  $$
    \rl_S\bigt{\Sing(X_t\xto{i_{X_T}}X_T)}\simeq i_{S*}i_{S}^*\rl_S\bigt{\Sing(X_t\xto{i_{X_T}}X_T)}
  $$
is supported on $s$. Consider now the diagram
\begin{equation}\label{crucial diagram}
  \begin{tikzcd}
    \rl_S\bigt{\Perf(X_t)}
    \arrow[r]
    &
    \rl_S\bigt{\Coh(X_t\xto{i_{X_t}}X_T)}
    \arrow[r]
    \arrow[d]
    \arrow[rd,"\xi"]
    &
    \rl_S\bigt{\Sing(X_t\xto{i_{X_t}}X_T)}
    \\
    p_*v_{X*}j_{X_T!}\Ql{,X_{\uL}}(\beta)
    \arrow[r]
    \arrow[rd,"\zeta"]
    &
    p_*v_{X*}\Ql{,X_T}(\beta)
    \arrow[r]
    \arrow[d]
    &
    p_*v_{X*}i_{X_T*}\Ql{,X_t}(\beta)
    \\
    &
    p_*v_{X*}j_{X_T*}\Ql{,X_{\uL}}(\beta)
    &
  \end{tikzcd}
\end{equation}
and observe that the two rows and the column in the middle are fiber-cofiber sequences. 
This has already been remarked for the first row. 
The second row is just localization in $\ell$-adic sheaves. 
As for the column in the middle, one observes that the map $p_*v_{X*}\bigt{\Ql{,X_T}(\beta)\to j_{X_T*}\Ql{,X_{\uL}}(\beta)}$ identifies with $\rl_S\bigt{\Perf(X_T)\xto{j_{X_T}^*}\Perf(X_{\uL})}$. 
The latter has fiber equal to $\rl_S\bigt{\Perf(X_T)_{X_t}}\simeq \rl_S\bigt{\Coh(X_t\xto{i_{X_T}}X_T)}$.

\medskip

Consider now the composition
  $$
    \rl_S\bigt{\Perf(X_t)}\to \rl_S\bigt{\Coh(X_t\xto{i_{X_T}}X_T)} \xto{\xi} p_*v_{X*}i_{X_T*}\Ql{,X_t}(\beta),
  $$
which we claim to be homotopic to zero. 
Indeed, the pushforwards along $X_t\to X_s$ and $X_T\to X$ induce a commutative diagram

\begin{equation}\label{diagram for equivariancy}
  \begin{tikzcd}
    \rl_S\bigt{\Perf(X_t)}
    \arrow[r]
    \arrow[d,"\simeq"]
    &
    \rl_S\bigt{\Coh(X_t\xto{i_{X_t}}X_T)}
    \arrow[r,"\rl_S(i_{X_T*})"]
    \arrow[d]
    &
    \rl_S\bigt{\Perf(X_T)}
    \arrow[r,"\rl_S(i_{X_T}^*)"]
    \arrow[d]
    &
    \underbrace{\rl_S\bigt{\Perf(X_t)}}_{\simeq p_*v_{X*}i_{X_T*}\Ql{,X_t}(\beta)}
    \arrow[d,"\simeq"]
    \\
    \rl_S\bigt{\Perf(X_s)}
    \arrow[r]
    &
    \rl_S\bigt{\Coh(X_s)}
    \arrow[r,"\rl_S(i_{X*})"]
    &
    \rl_S\bigt{\Perf(X)}
    \arrow[r,"\rl_S(i_{X}^*)"]
    &
    \underbrace{\rl_S\bigt{\Perf(X_s)}}_{\simeq p_*i_{X*}\Ql{,X_s}(\beta)}
  \end{tikzcd}
\end{equation}
where the vertical morphisms at the extremes are equivalences. By \cite[Lemma 3.26]{brtv18} the bottom composition is homotopic to zero and the claim follows.

Notice also that
  $$
    \coFib(\xi)\simeq \coFib(\zeta).
  $$
This is a general fact about diagrams like the one above in a stable $\oo$-category.

We now apply the octahedron construction to the composition
  $$
    \rl_S\bigt{\Perf(X_t)}\to \rl_S\bigt{\Coh(X_t\xto{i_{X_T}}X_T)} \xto{\xi} p_*v_{X*}i_{X_T*}\Ql{,X_t}(\beta)
  $$
and obtain the fiber-cofiber sequence
  $$
    \rl_S\bigt{\Sing(X_t\xto{i_{X_T}}X_T)}\to p_*v_{X*}i_{X_T*}\Ql{,X_t}(\beta)\oplus \rl_S\bigt{\Perf(X_t)}[1]\to \coFib(\zeta).
  $$
Observe now that all objects are supported on $s$ and that
  $$
    p_*v_{X*}i_{X_T*}\Ql{,X_t}(\beta)\oplus \rl_S\bigt{\Perf(X_t)}[1]\simeq p_*i_{X_*}\bigt{\Ql{,X_s}(\beta)\oplus \Ql{,X_s}(\beta)[1]}.
  $$
Moreover, by proper base change, we have
  $$
    i_{S*}i_S^*\bigt{\coFib(\zeta)}\simeq q_{T*}i_{X_T*}i_{X_T}^*j_{X_T*}\Ql{,X_{\uL}}(\beta).
  $$
We deduce that there is a fiber-cofiber sequence
  $$
\rl_S\bigt{\Sing(X_t\xto{i_{X_T}}X_T)}\to p_*i_{X_*}\bigt{\Ql{,X_s}(\beta)\oplus \Ql{,X_s}(\beta)[1]} \to q_{T*}i_{X_T*}i_{X_T}^*j_{X_T*}\Ql{,X_{\uL}}(\beta).
  $$
To conclude, observe that the term in the middle, which is equivalent to $\rl_S\bigt{\Perf(X_t)}\oplus \rl_S\bigt{\Perf(X_t)}[1]$, carries the trivial action of $\HL$ by \lemref{l-adic realization Perf(X_t)}. 
Therefore, it is equivalent to 
$$
p_*i_{X_*}\Ql{,X_s}(\beta)\otimes_{\Ql{,S}}\Ql{,S}^{\IL}
$$
by \lemref{I_L-invariant vanishing cycles}.
\end{proof}

\begin{notat}
We will denote the morphism appearing in \propref{l-adic realization of Sing(X_t-->X_T)} by
    $$
      \on{can}_{X_T}:p_*i_{X_*}\bigt{\Ql{,X_s}(\beta)\oplus \Ql{,X_s}(\beta)[1]} \to q_{T*}i_{X_T*}i_{X_T}^*j_{X_T*}\Ql{,X_{\uL}}(\beta).
    $$   
\end{notat}

\begin{rem}
The morphism $p_*i_{X_*}\Ql{,X_s}(\beta)\to q_{T*}i_{X_T*}i_{X_T}^*j_{X_T*}\Ql{,X_{\uL}}(\beta)$ obtained by restriction from the second map in the fiber-cofiber sequence of \propref{l-adic realization of Sing(X_t-->X_T)} corresponds to the one induced by the unit $\Ql{,X_T}(\beta)\to j_{X_T*}\Ql{,X_{\uL}}(\beta)$ under the equivalence
    $$
      p_*i_{X_*}\Ql{,X_s}(\beta)\simeq q_{T*}i_{X_T*}i_{X_T}^*\Ql{,X_T}(\beta).
    $$
In particular, as
    $$
      p_{s*}\bigt{\Ql{,X_s}(\beta)\oplus \Ql{,X_s}(\beta)[1]}\simeq p_{s*}\Ql{,X_s}(\beta)\otimes_{\Ql{,s}}\Ql{,s}^{\IL}
    $$
(see \cite[(4.3.43)]{brtv18}), we see that $\on{can}_{X_T}$ is obtained from the unit morphism $\Ql{,X_T}(\beta)\to j_{X_T*}\Ql{,X_{\uL}}(\beta)$ by recognizing that $q_{T*}i_{X_T*}i_{X_T}^*j_{X_T*}\Ql{,X_{\uL}}(\beta)$ has a natural $i_{S*}\Ql{,s}^{\IL}$-module structure.

In particular, we can write the fiber-cofiber sequence of \propref{l-adic realization of Sing(X_t-->X_T)} as
    $$
      \rl_S\bigt{\Sing(X_t\xto{i_{X_T}}X_T)}\to i_{S*}p_{s*}\Ql{,X_s}(\beta)\otimes_{\Ql{,s}}\Ql{,s}^{\IL} \xto{\on{can}_{X_T}} i_{S*}p_{s*}\Psi_p\bigt{\Ql{,X_{\uK}}(\beta)}^{\IL}.
    $$
\end{rem}

\begin{rem}
The diagram \eqref{diagram for equivariancy} also shows that the map 
    $$
      \on{can}_{X_T}:p_*i_{X_*}\bigt{\Ql{,X_s}(\beta)\oplus \Ql{,X_s}(\beta)[1]} \to q_{T*}i_{X_T*}i_{X_T}^*j_{X_T*}\Ql{,X_{\uL}}(\beta)
    $$
is $\HL$-equivariant. 
Indeed, combined with the version of \eqref{crucial diagram} for $X_T$ replaced by $X$, it implies that this map factors through 
    $$
      p_*i_X^*j_{X*}\Ql{,X_{\uK}}(\beta)\simeq \bigt{q_{T*}i_{X_T*}i_{X_T}^*j_{X_T*}\Ql{,X_{\uL}}(\beta)}^{\HL}.
    $$
\end{rem}

\sssec{Proof of \thmref{main thm A}}

We are finally ready to prove our first main theorem.
We will do so by showing that there is a homotopy of morphisms of algebras
$$
  \on{sp}^{\IL} \sim \on{can}_{X_T}:\underbrace{p_*i_{X_*}\bigt{\Ql{,X_s}(\beta)\oplus \Ql{,X_s}(\beta)[1]}}_{\simeq p_*i_{X_*}\Ql{,X_s}(\beta)\otimes_{\Ql{,S}}\Ql{,S}^{\IL}} \to q_{T*}i_{X_T*}i_{X_T}^*j_{X_T*}\Ql{,X_{\uL}}(\beta).
$$

Notice that $\on{sp}^{\IL}$ is homotopic to the morphism 
$$
  p_{*}i_{X*}\Ql{,X_s}(\beta)\otimes_{\Ql{,s}}\Ql{,s}^{\IL}\to q_{T*}i_{X_T*}i_{X_T}^*j_{X_T*}\Ql{,X_{\uL}}(\beta)
$$
which corresponds, under the adjunction $\bigt{-\otimes_{\Ql{,s}}\Ql{,s}^{\IL},\textup{Forget}}$, to the morphism
$$
p_{*}i_{X*}\Ql{,X_s}(\beta)\to q_{T*}i_{X_T*}i_{X_T}^*j_{X_T*}\Ql{,X_{\uL}}(\beta)
$$
induced by the adjunction $(j_{X_T}^*,j_{X_T*})$. This is \lemref{I_L-invariant vanishing cycles}.
The same is true for $\on{can}_{X_T}$, as one can see from diagram \ref{crucial diagram}.

\ssec{Non-commutative nature of \texorpdfstring{$\ell$}{l}-adic vanishing cycles}
In this subsection we prove \thmref{main thm B}.

\sssec{} 

Recall from \cite{brtv18} that the category of {\LG} models over $S$ is the ordinary category of pairs $(Y,f)$, where $Y$ is a flat $S$-scheme and $f:Y\to \bbA^{1}_{S}$ is a function. 
A morphism $(Y,f)\to (Z,g)$ between {\LG} models is a morphism of $S$-schemes $Y\to Z$ compatible with $f$ and $g$ in the obvious sense.

The assignment $(Y,f)\mapsto \Sing(Y_0\xto{i_Y}Y)$, where $i_Y:Y_0\to Y$ is the closed embedding of the fiber over zero of $f$ in $Y$, can be promoted to a functor
$$
  \Sing: \LG_S^\op \to \dgCat_S,
$$
where the transition maps are induced by pullbacks. See \cite[\S 2.3.15]{brtv18}.

\begin{notat}
    In this section, we will adopt the notation 
    $$
      \Sing(Y,f):=\Sing(Y_0\xto{i_Y}Y).
    $$
\end{notat}

\sssec{} 

Let $\scrE$ denote the filtered category of finite extensions
 $$
 \scrO_{\uK} \subseteq \scrO_{\uL}
 $$ 
of complete strict discrete valuations rings, like the one considered in \ref{notat: extensions dvr}.
For an $S$-scheme $Y$, we will denote by $(Y,\pi_{\uK})$ the {\LG} model over $S$ given by $Y$ with the function $Y\to S \xto{\pi_{\uK}} \bbA^{1}_{S}$.

\sssec{} 

Consider a proper flat regular $S$-scheme $X$ (generically smooth). 
For an extension $\scrO_{\uK} \subseteq \scrO_{\uL}$, let $X_{\scrO_{\uL}}$ denote the pullback $X\times_S \Spec(\scrO_{\uL})$. 
Then we get the following diagram of {\LG} models over $S$:
$$
  \scrE^\op \to \LG_S
$$
$$
  (\scrO_{\uK} \subseteq \scrO_{\uL})\mapsto (X_{\scrO_{\uL}},\pi_{\uK}).
$$
Notice that for every chain of extensions $\scrO_{\uK} \subseteq \scrO_{\uL} \subseteq \scrO_{\uM}$, the morphism of {\LG} models
$$
  (X_{\scrO_{\uM}},\pi_{\uK})\to (X_{\scrO_{\uL}},\pi_{\uK})
$$
is $\uH_{\uM}$-equivariant, where the $\uH_{\uM}$-action on $X_{\scrO_{\uL}}$ is induced by the quotient $\Gal(\uM/\uK)=\uH_{\uM} \to \HL=\Gal(\uL/\uK)$.

\sssec{} 

Composing this diagram with the functor $\Sing$, we get a diagram 
$$
  \ffd:\scrE \to \dgCat_S
$$
defined on objects by $(\scrO_{\uK} \subseteq \scrO_{\uL})\mapsto \Sing(X_{\scrO_{\uL}},\pi_{\uK})$.

\begin{rem}
The dg-category $\Sing(X_{\scrO_{\uL}},\pi_{\uK})$ is precisely the dg-category denoted by 
$$\Sing(X_t\xto{i_{X_T}}X_T)
$$ 
in the previous sections (for $T=\Spec(\scrO_{\uL})$).
\end{rem}

\sssec{}

It follows immediately from functoriality that each $\Sing(X_{\scrO_{\uL}},\pi_{\uK})$ carries a canonical $\HL$-action and that the dg functors
$$
  \Sing(X_{\scrO_{\uL}},\pi_{\uK})\to \Sing(X_{\scrO_{\uM}},\pi_{\uK})
$$
are compatible with these actions for every chain of extensions $\scrO_{\uK} \subseteq \scrO_{\uL} \subseteq \scrO_{\uM}$.

\sssec{} 

Recall that $\dgCat_S$ is a cocomplete $\oo$-category. 
We consider the colimit 
$$
  \ffS:=\varinjlim_{(\scrO_{\uK} \subseteq \scrO_{\uL})\in \scrE} \Sing(X_{\scrO_{\uL}},\pi_{\uK})
$$
of the diagram $\ffd$. 
It follows immediately that this dg-category carries a \emph{continuous} action of $\IK\simeq \varprojlim_{(\scrO_{\uK} \subseteq \scrO_{\uL})\in \scrE}\HL$. 
Roughly, this means that, for every object $A\in \ffS$, there exists some $(\scrO_{\uK} \subseteq \scrO_{\uL})\in \scrE$ such that $\IL\subseteq \IK$ acts trivially on the full subcategory $(A)\subseteq \ffS$ generated by $A$.

\sssec{Proof of \thmref{main thm B}} 

Notice that for every chain $\scrO_{\uK} \subseteq \scrO_{\uL} \subseteq \scrO_{\uM}$, there is a commutative diagram

\begin{equation*}
    \begin{tikzcd}[column sep=huge, row sep=large]
      \rl_S\bigt{\Sing(X_{\scrO_{\uL},\pi_{\uK}})}
      \arrow[r]
      \arrow[d,"\bigt{(X_{\scrO_{\uM}}\to X_{\scrO_{\uL}})\times_S s}^*"]
      &
      i_{S*}p_{s*}\Ql{,X_s}(\beta)^{\IL}
      \arrow[r,"\on{sp}^{\IL}"]
      \arrow[d]
      &
      i_{S*}\Psi_p\bigt{\Ql{,X_{\uK}}(\beta)}^{\IL}
      \arrow[d]
      \\
      \rl_S\bigt{\Sing(X_{\scrO_{\uM},\pi_{\uK}})}
      \arrow[r]
      &
      i_{S*}p_{s*}\Ql{,X_s}(\beta)^{\uI_{\uM}}
      \arrow[r,"\on{sp}^{\uI_{\uM}}"]
      &
      i_{S*}\Psi_p\bigt{\Ql{,X_{\uK}}(\beta)}^{\uI_{\uM}}
    \end{tikzcd}
\end{equation*}
where the rows are fiber-cofiber sequences of \thmref{main thm A} and the middle and rightmost vertical a morphisms are the canonical maps from $\IL$-homotopy fixed points to $\uI_{\uM}$-homotopy fixed points. Therefore, the filtered diagram
$$
  \rl_S \circ \ffd : \scrE \to \Mod_{\Ql{,S}(\beta)}\bigt{\shv_{\Ql{}}(S)}
$$
is equivalent to the filtered diagram
$$
  \scrE \to \Mod_{\Ql{,S}(\beta)}\bigt{\shv_{\Ql{}}(S)}
$$
$$
  (\scrO_{\uK} \subseteq \scrO_{\uL}) \mapsto i_{S*}\uH^*_{\et} \bigt{X_s,\Phi_p(\Ql{,S}(\beta))}^{\IL}[-1].
$$

\sssec{}
Recall that $\rl_S$ commutes with filtered colimits. 
Since the equivalences 
$$
  \rl_S\bigt{\Sing(X_{\scrO_{\uL}},\pi_{\uK})}\simeq i_{S*}\uH^*_{\et} \bigt{X_s,\Phi_p(\Ql{,S}(\beta))}^{\IL}[-1]
$$ 
are compatible with the $\HL$-actions, we get that
\begin{eqnarray}
\nonumber
   \rl_S(\ffS) & = & \rl_S \bigt{\varinjlim_{(\scrO_{\uK} \subseteq \scrO_{\uL})\in \scrE} \Sing(X_{\scrO_{\uL}},\pi_{\uK})} \hspace{0.5cm} (\text{def. of } \ffS)
   \\
   \nonumber
   & \simeq & \varinjlim_{(\scrO_{\uK} \subseteq \scrO_{\uL})\in \scrE}\rl_S \bigt{\Sing(X_{\scrO_{\uL}},\pi_{\uK})} \hspace{0.5cm} (\rl_S \text{ commutes with filtered colimits})
   \\
   \nonumber
   & \simeq & \varinjlim_{(\scrO_{\uK} \subseteq \scrO_{\uL})\in \scrE}i_{S*}\uH^*_{\et} \bigt{X_s,\Phi_p(\Ql{,S}(\beta))}^{\IL}[-1] \hspace{0.5cm} (\text{\thmref{main thm A}})
   \\
   \nonumber
   & \simeq & i_{S*} \uH^*_{\et}\bigt{X_s,\Phi_p(\Ql{,X}(\beta))}[-1] \hspace{0.5cm} (\text{continuity of the action of } \IK).
\end{eqnarray}
This concludes the proof of Theorem \ref{main thm B}.

\appendix
\section{Remarks on the properness hypothesis}\label{sec:prop hyp}

In this final section, we briefly comment on the properness hypothesis for the morphism $p:X\to S$. This assumption is superfluous, provided that one is willing to work at the level of $\ell$-adic sheaves on $X$. 
This observation is the analog of \cite[Footnote $8$, page 503]{tv22} in the case where Galois actions are taken into account.

\ssec{}

An attentive reader might have noticed that the properness hypothesis is used only once throughout the paper: in the proof of \propref{l-adic realization of Sing(X_t-->X_T)} in order to invoke proper base-change. 
This is needed because we work with the $\ell$-adic realization functor $\rl_S$.
However, as explained in \cite[Remark 2.2.2]{tv22}, $\rl_S$ admits a relative version
$$
  \rl_X:
  \dgCat_X
  \to
  \Mod_{\Ql{,X}(\beta)}\bigt{\Shv(X)}
$$
with the same properties of $\rl_S$.
The computations and the proofs in this paper all work \emph{mutatis mutandis} by applying $\rl_X$ in place of $\rl_S$.
Only the fact that 
$$
  \Coh(X_t\xto{i_{X_T}}X_T)
  \hto
  \Coh(X_t)
  \to
  \Sing(X_T)
$$
is a localization sequence of $X$-linear dg-categories deserves a bit of explanation.

\ssec{}

The functor
$$
  p_*:
  \dgCat_X
  \to
  \dgCat_S
$$
admits a (symmetric monoidal) left adjoint 
$$
  p^*=-\otimes_{\Perf(S)}\Perf(X):
  \dgCat_S
  \to
  \dgCat_X
$$
which preserves localization sequences.

\ssec{}

By \corref{useful loc seq}, we obtain that
$$
  \bigt{\Coh(G_t\xto{a_1*}t)
  \hto
  \Coh(G_t)
  \to
  \Sing(t)}
  \otimes_{\Perf(S)}\Perf(X)
$$
is a localization sequence of (left) $\sB_X^+$-modules, where $\sB_X^+:=\sB^+\otimes_{\Perf(S)}\Perf(X)$.

\ssec{}

Clearly, $\Coh(X_s)$ (regarded as an $X$-linear dg-category) admits a left $\sB_X^+$-module structure. 
As a consequence, $\Coh(X_s)^{\op}$ admits a right $\sB_X^+$-module structure.

\ssec{}
One sees that 
$$
  \Coh(X_s)^{\op}
  \otimes_{\sB_X^+}
  \Bigt{
  \bigt{\Coh(G_t\xto{a_1*}t)
  \hto
  \Coh(G_t)
  \to
  \Sing(t)}
  \otimes_{\Perf(S)}\Perf(X)
  }
$$
identifies with the diagram of $X$-linear dg-categories
\begin{equation}\label{useful loc seq 2/X}
  \Coh(X_t\xto{i_{X_T*}}X_T)
  \hto
  \Coh(X_t)
  \to
  \Sing(X_T).
\end{equation}
In particular, this is a localization sequence in $\dgCat_X$ (the proofs of \lemref{Coh(X_s)^op otimes_CB^+ - preserves loc seq} and \propref{useful loc seq 2} can be adapted easily to the $X$-linear situation).

\begin{rmk}
If we apply the forgetful functor $p_*:\dgCat_X \to \dgCat_S$ to \eqref{useful loc seq 2/X}, we find the localization sequence of $S$-linear dg-categories obtained in \propref{useful loc seq 2}.
\end{rmk}

\ssec{}

Given this key ingredient, the proofs of the computations of motivic and $\ell$-adic realizations in the main body of the paper apply \emph{before taking $p_*$}. Hence, we can avoid any reference to proper base-change and in particular we obtain the following:
\begin{unthm}
    Let $p:X\to S$ be a flat and generically smooth morphism of finite type. Assume that $X$ is regular.
    There is an equivalence of $i_{X*}\Ql{,X}^\IL(\beta)$-modules 
    $$
      \rl_X\bigt{\Sing(X_t\xto{i_{X_T}}X_T)}
      \simeq
      i_{X*}\Phi_p\bigt{\Ql{,X}(\beta)}^{\IL}[-1],
    $$
    compatible with the natural $\GLK$-actions.
\end{unthm}

\printbibliography

\end{document}